\documentclass[11pt]{article} 
\usepackage{amsmath}
\usepackage{amssymb}
\usepackage{theorem}
\usepackage{euscript}
\usepackage{pstricks}
\usepackage{authblk}
\usepackage{verbatim}
\usepackage{hyperref}
\hypersetup
{
	colorlinks,
	citecolor=black,
	filecolor=black,
	linkcolor=black,
	urlcolor=blue
}
\hypersetup{linktocpage}
\newtheorem{theorem}{Theorem}[section]
\newtheorem{lemma}{Lemma}[section]
\newtheorem{corollary}{Corollary}[section]
\numberwithin{equation}{section}
\newtheorem{assumption}[theorem]{Assumption}

\theoremstyle{definition}
 
\theoremstyle{remark}

\newcommand{\brac}[1]{\left(#1\right)}

\newcommand{\brab}[1]{\left\{#1\right\}}
\newcommand{\set}[2]{\{#1\,:\,#2\}}

\newcommand{\norm}[2]{\left\|{#1}\right\|_{#2}}

\newcommand{\ba}{{\boldsymbol{a}}}
\newcommand{\bb}{{\boldsymbol{b}}}

\newcommand{\bk}{{\boldsymbol{k}}}

\newcommand{\bs}{{\boldsymbol{s}}}

\newcommand{\bx}{{\boldsymbol{x}}}

\newcommand{\by}{{\boldsymbol{y}}}

\newcommand{\bz}{{\boldsymbol{z}}}

\newcommand{\bgamma}{{\boldsymbol{\gamma}}}
\newcommand{\brho}{{\boldsymbol{\rho}}}
\newcommand{\bvarrho}{{\boldsymbol{\varrho}}}
\newcommand{\bsigma}{{\boldsymbol{\sigma}}}

\newcommand{\balpha}{{\boldsymbol{\alpha}}}
\newcommand{\bmu}{{\boldsymbol{\mu}}}

\newcommand{\rd}{{\rm d}}

\def\II{\mathbb I}

\def\ZZd{{\mathbb Z}^d}

\def\ZZ{{\mathbb Z}}

\def\RR{{\mathbb R}}

\def\NN{{\mathbb N}}

\def\II{{\mathbb I}}
\def\CC{{\mathbb C}}
\def\CC{{\mathbb C}}

\def\NN{{\mathbb N}}
\def\RR{{\mathbb R}}

\def\Vv{{\mathbb P}}
\def\UU{{\mathbb U}}
\def\Vv{{\mathcal P}}

\def\FF{{\mathcal F}}
\def\Bb{{\mathcal B}}

\def\ZZd{{\mathbb Z}^d}

\def\IIi{{\mathbb I}^\infty}

\def\CCi{{\mathbb C}^\infty}

\def\RRi{{\mathbb R}^\infty}

\def\UUi{{\mathbb U}^{\infty}}

\def\Bb{{\mathcal B}}

\def\FF{{\mathcal F}}

\def\Jj{{\mathcal J}}

\def\Ss{{\mathcal S}}

\def\Vv{{\mathcal V}}

\def\II{{\mathbb I}}
\def\CC{{\mathbb C}}
\def\ZZ{{\mathbb Z}}
\def\NN{{\mathbb N}}
\def\RRR{{\mathbb R}}

\def\FF{{\mathbb F}}

\def\RRRi{{\mathbb R}^\infty}

\def\supp{\operatorname{supp}}
\def\dv{\operatorname{div}}

\title{\sffamily Sampling recovery in Bochner spaces and applications to \\ 
	 parametric PDEs}

\author[a]{Felix Bartel}
\affil[a]{{Mathematisch-Geographische Fakult\"at, KU Eichst\"att-Ingolstadt
	\protect\\
     85270 Eichst\"att, Germany
	\protect\\
    Email: felix.bartel@ku.de}}
	
\author[b]{Dinh D\~ung\footnote{Corresponding author}}
\affil[b]{Information Technology Institute, Vietnam National University, Hanoi
	\protect\\
	144 Xuan Thuy, Cau Giay, Hanoi, Vietnam
	\protect\\
	Email: dinhzung@gmail.com}

\date{\today}
\begin{document}
\maketitle

\begin{abstract}	
    We prove convergence rates of linear sampling recovery of functions in abstract Bochner spaces satisfying weighted summability of their generalized polynomial chaos expansion coefficients.
    The underlying algorithm is a function-valued extension of the least squares method widely used and thoroughly studied in scalar-valued function recovery.
    We apply our theory to two core problems in Computational Uncertainty Quantification.
    First, we address
    non-intrusive approximations of solutions to
    parametric elliptic or parabolic PDEs with log-normal or affine inputs using a
    finite set of particular solvers.
    Second, we consider approximating infinite-dimensional holomorphic functions   that arise as solutions to more general parametric PDEs with Gaussian random field inputs.
    Our framework allows a unified treatment of the log-normal and affine input models and yields substantial improvements to the state of the art of these problems.
    More specifically, we obtain convergence rates that improve known results by a polynomial factor in the log-normal case and by a logarithmic factor in the affine case.
 
	
\end{abstract}

\section{Introduction and main results}\label{sec:intro}

In Computational Uncertainty Quantification, the problem of efficient approximation for parametric PDEs with random inputs is of great interest and significant progress was achieved in recent years.
The number of works on this topic is too large to mention all of them.
We refer the reader to \cite{Adcock2022,CoDe15a,DNSZ2023,GWZ2014,SG2011} for surveys and further references on various aspects of it.

{The principal distinction among numerical methods for parametric PDEs is whether they are non-intrusive or intrusive. Non-intrusive methods use an existing solver (exact or approximate) for the PDE, enabling deployment even when the solver is treated as a black box. Intrusive methods, by contrast, incorporate the explicit PDE formulation into the approximation process, demanding complete knowledge of the parametric PDE model.
For a detailed comment, see, e.g., \cite[Section 1.6]{CoDe15a}.}

We are interested in the problem of linear semi-discrete {non-intrusive} approximation of solutions $u(\by)$ to parametric PDEs and the corresponding convergence rate based on a finite number of particular solvers $u(\by_1),..., u(\by_n)$ for $\by_1,..., \by_n$ from a domain $\UUi$, where $\UUi$  usually is the infinite-dimensional domain $\IIi := [-1,1]^\infty$ or $\RRi$.
We do not consider fully discrete {(multilevel)  approximations} which simultaneously treat a discretization on both spatial and parametric domains.
The reader can consult \cite{Dung21,DNSZ2023} for surveys and further references in this direction.
In the present paper, we focus on parametric elliptic and parabolic PDEs with log-normal or affine random diffusion coefficients.
Related problems of adaptive nonlinear semi-discrete  approximation for parametric PDEs were investigated in \cite{CCS15,CoDe15a,ABDM2024,CCMNT2015,CCS13}, and of linear semi-discrete  approximation in 
\cite{ADM2024,BNT2007,Dung19,Dung21,DD-Erratum23, EST18,MNST2014,NTW2008,NTW2008a,DNSZ2023,ZDS19,ZS20}.

Solutions $u(\by)$, $\by \in \UUi$, to parametric PDEs can be considered as elements of a Bochner space $L_2(\UUi,X;\bmu)$, where
$X$ is a Hilbert space,  and $\bmu$ is a probability measure on $\UUi$. The above approximation problem is equivalent to solving a linear sampling recovery problem in the specified space.
We leverage recent breakthrough results for scalar-valued sampling recovery in reproducing kernel Hilbert spaces (RKHS) \cite{BSU23, DKU2023, KUV21, KU21a}, and transfer them to the Hilbert-valued approximation setting (see surveys and further references on related results in \cite{DTU18B, SU2023}). 
This transferring mechanism  is basically the same as the standard  “lifting” method
  of least squares approximation studied in 
\cite{ABDM2023,ABDM2024,ADM2024, CCMNT2015, CM2018}. The novelty of our approach lies in its systematic nature. Rather than lifting individual algorithms from the scalar-valued to the Hilbert-valued setting, we instead transfer the approximation properties of general Bochner spaces -- as captured by the sampling 
$n$-width -- from the scalar-valued to the Hilbert-valued level, in the sense of the identity \eqref{rho_n=Introduction} below. This constitutes the central contribution of the present paper.

Before presenting the main result on sampling recovery in Bochner spaces, we need to comment on the {setup}.
Let $(U,\Sigma,\mu)$ be a probability measure space with $U$ being a separable topological space and let $X$ be a complex separable Hilbert space.
Denote by $L_2(U,X;\mu) = L_2(U,\CC;\mu) \otimes X$ the Bochner space of strongly $\mu$-measurable mappings $v$ from $U$ to $X$, equipped with the norm
\begin{equation} \nonumber
	\|v\|_{L_2(U,X;\mu)}
	:= \
	\left(\int_{U} \|v(\by)\|_X^2 \, \rd \mu(\by) \right)^{1/2}.
\end{equation}
Notice that because $U$ and $X$ are separable, so is $L_2(U,X;\mu)$.
We fix $\brac{\varphi_s}_{s \in \NN}$, an orthonormal basis of $L_2(U,\CC;\mu)$.
Then a function $v \in L_2(U,X;\mu)$ can be represented by the expansion
\begin{equation} \label{series}
	v(\by)=\sum_{s \in \NN} v_s \,\varphi_s(\by)
 \quad\text{with}\quad
	v_s:=\int_U v(\by)\,\overline{\varphi_s(\by)}\, \rd\mu (\by) \in X
\end{equation}
with convergence in $L_2(U,X;\mu)$.
Moreover, for every $v \in L_2(U,X;\mu)$ represented by the 
series \eqref{series}, Parseval's identity holds,
\begin{equation} \nonumber
	\|v\|_{L_2(U,X; \mu)}^2
	\ = \ \sum_{{s\in \NN}} \|v_s\|_X^2.
\end{equation}

Throughout the present paper, we fix $\brac{\sigma_s}_{s \in \NN}$, a non-decreasing sequence of positive numbers such that 
$\bsigma^{-1}:=\brac{\sigma_s^{-1}}_{s \in \NN} \in \ell_2(\NN)$.
For given $U$ and $\mu$, denote by $H_{X,\bsigma}$ the linear subspace in $L_2(U,X;\mu)$ of all $v$ such that the norm
\begin{equation} \nonumber
	\|v\|_{H_{X,\bsigma}}
	:= \
	\brac{\sum_{s \in \NN} \brac{\sigma_s \|v_s\|_X}^2}^{1/2} < \infty.		
\end{equation}
	In particular, the space $H_{\CC,\bsigma}$ is the linear subspace in $L_2(U,\CC;\mu)$ equipped with its own inner product
\begin{equation} \nonumber
	\langle f,g \rangle_{H_{\CC,\bsigma}}
	:= \
	\sum_{s \in \NN} \sigma_s^2 
	\langle f,\varphi_s \rangle_{L_2(U,\CC;\mu)}
	\overline{\langle g,\varphi_s \rangle_{L_2(U,\CC;\mu)}}
\end{equation}
and forms a reproducing kernel Hilbert space with the reproducing kernel
\begin{equation} \nonumber
	K(\cdot,\by)
	:= \
	\sum_{s \in \NN} \sigma_s^{-2} \varphi_s(\cdot)\overline{\varphi_s(\by)}
\end{equation}
with eigenfunctions $\brac{\varphi_s}_{s \in \NN}$ and eigenvalues 
$\brac{\sigma_s^{-1}}_{s \in \NN}$.
Moreover, $K$ satisfies the finite trace property
\begin{equation*}\label{trace assumption}
	\int_U K(\bx,\bx) \rd \mu(\bx) \ < \ \infty,
\end{equation*}
which is not only natural in sampling recovery but also crucial to the techniques we employ.  {We refer  readers, e.g.,  to \cite{BT2004,MU2021} and references therein for necessary background on RKHSs.}

{We study the approximate recovery of functions in the space $H_{X,\bsigma}$
	from a finite set of their samples. To ensure a coherent formulation of the problem, we begin with a preliminary observation.
	From the separability of $X$ it follows that there exists  a set $U_0 \subset U$ satisfying  $\mu(U\setminus U_0)=0$
such that
\begin{equation} \nonumber
	K(\bx,\by)
	:= \
	\sum_{s \in \NN} \sigma_s^{-2} \varphi_s(\bx)\overline{\varphi_s(\by)}, \ \ \forall \bx, \by \in U_0,
\end{equation}
and
\begin{equation} \nonumber
	f(\by)
	\ = \
	\sum_{s \in \NN} 
	\sigma_s^{-2}
	\langle f,\varphi_s \rangle_{H_{\CC,\bsigma}}
	\varphi_s(\by), \ \ 
    \forall  f \in H_{\CC,\bsigma},  \ \forall  \by \in U_0.
\end{equation}
This means that the pointwise evaluations $f(\by)$ are well-defined for every $\by$ belonging the full measure subset $U_0$ of $U$. 
Let $\brac{\psi_j}_{j \in \NN}$ be an orthonormal basis of  $X$. 
If $v \in H_{X,\bsigma}$, then $\langle v,\psi_j \rangle_X \in H_{\CC,\bsigma}$ for every $j$. 
Consequently, the pointwise evaluations $\langle v(\by),\psi_j \rangle_X $ are well-defined for every
$\by \in U_0$. This implies that the pointwise evaluations  $v(\by)$ are also well-defined for every
$\by \in U_0$.

Throughout this paper, in the context of sampling recovery,  the inclusion $v \in H_{X,\bsigma}$
means that $v$ is a representative of an element from $L_2(U,X;\mu)$ whose pointwise evaluations $v(\by)$ for every
$\by \in U_0$ are well-defined.
 
Let us  formulate the problem of  sampling recovery for 
 $v \in H_{X,\bsigma}$ as follows:
The goal is to find a sample $y_1,\dots,y_n\in U_0$ and functions $h_1, \dots,h_n\in L_2(U,\CC;\mu)$ such that $v$ is approximated from its values $v(\by_1),\ldots, v(\by_n)$ by the linear sampling algorithm $S_n^X$ on $U$ defined as
\begin{equation} \label{S_k}
	(S_n^X v)(\by): = \sum_{i=1}^n v(\by_i) h_i(\by).
\end{equation}
For convenience, we assume that some of the sample points $\by_i$ may coincide. 

    Denote by
\begin{equation*} \label{B_X^q}
	B_{X,\bsigma}:= \brab{v \in H_{X,\bsigma}: \, \norm{v}{H_{X,\bsigma}} \le 1}
\end{equation*}
the unit ball in $H_{X,\bsigma}$.

Let $\Ss_n^X$ be the family of all linear sampling algorithms 
$S_k^X$ in $L_2(U,X;\mu)$ of the form \eqref{S_k} with $k \le n$.
 To study the optimality of linear sampling algorithms from 
 $\Ss_n^X$ for the set $B_{X,\bsigma}$ 
 and their convergence rates we use the (linear) sampling $n$-width
\begin{equation*} \label{rho_n}
	\varrho_n(B_{X,\bsigma}, L_2(U,X;\mu)) :=\inf_{S_n^X \in \Ss_n^X} \ \sup_{v\in B_{X,\bsigma}} 
	\|v - S_n^X v\|_{L_2(U,X;\mu)}.
\end{equation*}

Let us describe the main contribution of the present paper.}
We establish convergence rates for an extension of the least squares method with varying sampling strategies with the following basic result.
For the sampling algorithm $S_n^X$ in $L_2(U,X;\mu)$ defined by \eqref{S_k}, we have
\begin{equation*} \label{sampling-equality}
	\sup_{v \in B_{X,\bsigma}} \norm{v- S_n^X v}{L_2(U,X;\mu)}
	= \sup_{f \in B_{\CC,\bsigma}} \norm{f - S_n^\CC f}{L_2(U,\CC;\mu)},
\end{equation*}
and, hence,
\begin{equation} \label{rho_n=Introduction}
    \varrho_n(B_{X,\bsigma}, L_2(U,X;\mu)) \ = \ \varrho_n(B_{\CC,\bsigma}, L_2(U,\CC;\mu)).
\end{equation}
This relation makes available bounds on the sampling widths in the classical Lebesgue space $L_2(U,\CC;\mu)$ applicable to a general Bochner space $L_2(U,X;\mu)$.

From the equality \eqref{rho_n=Introduction} and an inequality between the sampling widths and Kolmogorov widths proven in \cite[Theorem 1]{DKU2023} we derive the  optimal convergence rate of $\varrho_n(B_{X,\bsigma}, L_2(U,X;\mu))$ for $0< q<2$ in the sense that
the relations 
{
	\begin{equation} \label{rho_n><Introduction}
        (n+1)^{-1/q}
\ \le \
\sup_{{\bsigma}: \ \|\bsigma^{-1}\|_{\ell_q(\NN) \le 1}}	\varrho_n(B_{X,\bsigma}, L_2(U,X;\mu)) 
        \ \lesssim \ n^{-1/q}
\end{equation}
hold. In particular, \eqref{rho_n><Introduction} holds for the  Bochner space $L_2(\UUi,X;\bmu)$ with $\bmu$ denoting the infinite tensor-product standard Gaussian measure on  $\UUi=\RRi$, or the infinite tensor-product Jacobi probability measure on $\UUi=\IIi$.}
It is worth mentioning that the underlying sampling algorithm attaining the convergence rate in \eqref{rho_n><Introduction} is an extension to Bochner spaces of a classical least squares approximation with a non-constructive subsampling.
Regarding the constructiveness of linear sampling algorithms, similar extensions of a classical least squares approximation and of a least squares approximation with a special constructive subsampling give the error bounds $n^{-1/q}(\log n)^{1/q}$ and $ n^{-1/q}(\log n)^{1/2}$, respectively.
Thanks to this constructive subsampling, the cost of computation is significantly reduced for a sufficiently large number of sample points (for details, see \cite{BSU23}).

We apply these approximation results in general Bochner spaces to the previously formulated linear approximation of solutions to parametric PDEs with affine or log-normal inputs as well as to the approximation of infinite-dimensional holomorphic functions, significantly improving existing convergence rates.
Moreover, differing from the previous papers mentioned above, which considered the affine and log-normal cases of random inputs separately, with this approach we treat both cases together by employing a \emph{unified method}.

The specific setting for the linear  approximation problem for a wide class of parametric PDEs with random inputs as well as for infinite-dimensional holomorphic functions is as follows.
Under a certain condition the weak parametric solution $u(\by)$ to a parametric elliptic PDE equation with log-normal ($\UUi=\RRi$) or affine 
($\UUi=\IIi$) random inputs satisfies a weighted $\ell_2$-summability of the energy norms of the Hermite or Jacobi generalized polynomial chaos (GPC) expansion coefficients, respectively, in terms of the inclusion 
$u(\by) \in M B_{V,\bsigma}$ with $\norm{\bsigma^{-1}}{\ell_q} \le N$ for some $0<q<2$, $M,N >0$ and {positive sequences $\bsigma$}, where $V:= H_0^1(D)$ is the energy space and $D$ is the spatial domain (see Lemmata~\ref{lemma:weighted summablity,lognormal} and \ref{lemma:weighted summablity,affine} below).
This allows us to apply all the above results for abstract Bochner spaces to parametric elliptic or parabolic PDEs.
The most significant application is that, by \eqref{rho_n><Introduction},  there exists a linear sampling algorithm $S_n^V$ in $L_2(\UUi,V;\bmu)$ of the form \eqref{S_k} for $X=V$ such that
\begin{equation*} \label{u -S_n u-Introduction}
	\|u -S_n^V u\|_{L_2(\UUi,V;\bmu)} \leq C MN n^{-1/q},
\end{equation*}
where $C$ is a positive constant independent of $M,N,n$ and $u$.
For log-normal random inputs ($\UUi= \RRi$), the convergence rate of linear non-intrusive approximation of the parametric solution $u(\by)$ obtained via the sampling algorithm $S_n^V$ is $n^{-1/q}$.
This is an important improvement compared to all previous works, such as \cite{BNT2007,Dung19,Dung21,DD-Erratum23,DNSZ2023,EST18,MNST2014,NTW2008,NTW2008a}, which are off by a factor $n^{1/2}$ from this convergence rate.
In particular, it is significantly better than the rates 
$n^{- \frac{1}{2}(1/q- 1/2)}$ and $n^{-(1/q - 1/2)}$ which have been recently obtained in
\cite[Theorem 3.18]{EST18} and \cite[Corollary 5.9]{Dung21}, respectively.
The same notable improvement in the convergence rate holds true for linear polynomial interpolation of relevant infinite-dimensional holomorphic functions on $\RRi$ (cf.\ \cite{DNSZ2023})  and of parametric parabolic PDEs.
In the case of affine random inputs ($\UUi=\IIi$), the convergence rate $n^{-1/q}$ is better than the best-known convergence rate $(n/\log n)^{-1/q}$ of (linear and non-linear) non-intrusive approximation (cf.\ \cite{ADM2024,CCMNT2015,Dung19,ZDS19}).
Notice also that the convergence rate $n^{-1/q}$ coincides with the optimal convergence rate of (linear and non-linear) intrusive spectral and Galerkin approximation of solutions to parametric PDEs with random inputs (cf.\ \cite{BCDC17,BCDM17,BCM17,Dung21}).
We believe that, using the techniques developed in this paper, similar improvements could also be achieved for the linear approximation of relevant infinite-dimensional holomorphic functions on $\IIi$.
However, this is beyond the scope of current consideration.
Last but not least, it is worth emphasizing that the proof methods for the main results are simple yet entirely novel.
The approaches rely on the connection between sampling recovery in abstract RKHSs and {associated approximations} of solutions to parametric PDEs.
The proof methods { are then based} on the weighted summability properties of the GPC expansion coefficients.
Moreover, these methods {can be extended and generalized} to other problems in Uncertainty Quantification for parametric PDEs with random inputs, such as fully discrete multilevel non-intrusive approximation, among others.
{The last problem has been recently considered in~\cite{DD2025}.}

Note that in all convergence rates reported in this work, the  sequence 
$\bsigma$
is assumed to be known.
It determines the selected basis $(\varphi_s)_{s \in \NN}$ and, consequently, the least squares sampling algorithm, which is ``transferred'' to Hilbert-valued functions.
The case in which 
$\bsigma$
is unknown has been treated in \cite{ADM2024}, wherein the problem is approached via compressed-sensing techniques.

The rest of the paper is organized as follows. In Section \ref{Sampling recovery in Bochner spaces}, we investigate sampling recovery in abstract Bochner spaces, in particular, with infinite-dimensional measure. Here, we present some least squares methods and their extensions to Bochner spaces. In Section 
\ref{Applications to parametric PDEs with random inputs} and 
\ref{Applications to holomorphic functions}, we apply the results of Section 
\ref{Sampling recovery in Bochner spaces} to linear  approximation of solutions to parametric elliptic or parabolic PDEs with affine 
or log-normal {random inputs, and for} infinite-dimensional holomorphic functions on $\RRi$, respectively.
In Section \ref{Constructiveness and alternative least squares methods}, we discuss constructiveness and alternative sampling methods which can be applied to {non-intrusive} approximations for parametric PDEs and infinite-dimensional holomorphic functions.

\medskip
\noindent
{\bf Notations:} \ As usual, $\NN$ denotes the natural numbers, $\ZZ$ the integers, $\RR$ the real numbers, $\CC$ the complex numbers, and 
$ \NN_0:= \{s \in \ZZ: s \ge 0 \}$.
We denote by $\RR^\infty$ and $\IIi:=[-1,1]^\infty$ the
sets of all sequences $\by = (y_j)_{j\in \NN}$ with $y_j\in \RR$ and $y_j\in [-1,1]$, respectively, and by $\CCi$ the
set of all sequences $\bz = (z_j)_{j\in \NN}$ with $z_j\in \CC$.
 Denote by $\FF$ the set of all sequences of non-negative integers $\bs=(s_j)_{j \in \NN}$ such that their support $\supp (\bs):= \{j \in \NN: s_j >0\}$ is a finite set. 
If $\ba= (a_j)_{j \in \Jj}$ is a set of positive numbers with any index set $\Jj$, then we use the notation 
$\ba^{-1}:= (a_j^{-1})_{j \in \Jj}$.
For $\bs, \bs' \in \FF$ and $\by \in \RRi$, we write: 
$\bs!:= \prod_{j \in \NN}s_j!$, \ $\by^\bs:= \prod_{j \in \NN}y_j^{s_j}$, \
$\binom{\bs}{\bs'}:= \prod_{j \in \NN}\binom{s_j}{s'_j}$.
We use the letter $C$ to denote general 
positive constants which may take different values, and $C_{a,b,...}$ to denote a positive constant depending on $a,b,...$.
 For the quantities $A_n(f,\bk)$ and $B_n(f,\bk)$ depending on 
 $n \in \NN$, $f \in W$, $\bk \in \ZZd$, 
 we write $A_n(f,\bk) \lesssim B_n(f,\bk)$, $f \in W$, $\bk \in \ZZd$ ($n \in \NN$ is dropped)
 if there exists some constant $C >0$ such that 
 $A_n(f,\bk) \le CB_n(f,\bk)$ for all $n \in \NN$, $f \in W$, $\bk \in \ZZd$ (the notation $A_n(f,\bk) \gtrsim B_n(f,\bk)$ has the obvious opposite meaning), and 
 $A_n(f,\bk) \asymp B_n(f,\bk)$ if $A_n(f,\bk) \lesssim B_n(f,\bk)$
 and $B_n(f,\bk) \lesssim A_n(f,\bk)$. We denote by $|G|$ the cardinality of the set $G$. 
 
\section{Sampling recovery in Bochner spaces}
\label{Sampling recovery in Bochner spaces}	

In this section, we show that the problem of linear sampling recovery of functions in the space $H_{X,\bsigma}$ for a general separable Hilbert space $X$ can be reduced to the particular case of the RKHS	$H_{\CC,\bsigma}$. This allows, in particular, to extend linear least squares sampling algorithms in $H_{\CC,\bsigma}$ to $H_{X,\bsigma}$ while preserving the accuracy of approximation. Hence, we are able to derive convergence rates of various extended linear least squares sampling algorithms
for functions in $B_{X,\bsigma}$ based on some recent
 results on inequalities between sampling widths and Kolmogorov widths of the unit ball $B_{\CC,\bsigma}$ which are realized by linear least squares sampling algorithms.
	
\subsection{Extension of least squares approximation to Bochner spaces}
Recall that $\brac{\varphi_s}_{s \in \NN}$ is {a fixed orthonormal} basis of $L_2(U,\CC;\mu)$ and for a function $v \in L_2(U,X;\mu)$  {the expansion \eqref{series} holds, with coefficients 
$\brac{v_s}_{s \in \NN}$} defined as in \eqref{series}.	We will need the following auxiliary result.		
Let $A^X$ be a general linear operator in $L_2(U,X;\mu)$ defined for $v\in L_2(U,X;\mu)$ by 	
\begin{equation} \label{eq:A}
    {A^Xv := 
    \sum_{k \in \NN} \brac{\sum_{s \in \NN} a_{k,s}	v_s}\varphi_k,}
\end{equation}
where $(a_{k,s})_{(k,s) \in \NN^2}$	 is an infinite-dimensional matrix. 

\begin{lemma} \label{lemma:norms-linear-operators}
    Let $A^X$ be linear and bounded.
    We have
	\begin{equation} \nonumber
		\big\|A^X\big\|_{H_{X,\bsigma}\to L_2(U,X;\mu)}
		=
		\big\|A^{\CC}\big\|_{H_{\CC,\bsigma} \to L_2(U,\CC;\mu)}.
	\end{equation}
\end{lemma}

\begin{proof}
    Let $\brac{\psi_j}_{j \in \NN}$ be an orthonormal basis of $X$.
    We then have the representation of $v\in H_{X,\bsigma}$ with respect to the basis of $X$:
    \begin{equation*}
        v
        = \sum_{s\in\NN} v_s \varphi_s
        = \sum_{j\in\NN} \Big(\sum_{s\in\NN} \langle v_s, \psi_j\rangle_X \varphi_s\Big) \psi_j
        = \sum_{j\in\NN} \underbrace{\langle v, \psi_j\rangle_X}_{\in H_{\CC,\bsigma}} \psi_j .
    \end{equation*}
    Next we revert the definition of the Hilbert-valued $A^X$ to the scalar-valued $A^\CC$:
    \begin{align*}
        A^Xv
        &= \sum_{k\in\NN} \Big(\sum_{s\in\NN} a_{k,s} \sum_{j\in\NN} \langle v_s, \psi_j\rangle_X \psi_j \Big) \varphi_k \\
        &= \sum_{j\in\NN} \Big( \sum_{k\in\NN}  \Big(\sum_{s\in\NN} a_{k,s} \langle v_s, \psi_j\rangle_X \Big) \varphi_{k} \Big) \psi_j
        = \sum_{j\in\NN} (A^\CC \langle v, \psi_j\rangle_X) \psi_j .
      \end{align*}
      For $\tilde v_j = \langle v,\psi_j\rangle_X$ we have by Parseval's identity
    \begin{align*}
        \|A^Xv\|_{L^2(U, X; \mu)}^2
        &= \sum_{j\in\mathbb N}\|A^{\CC}\tilde v_j\|_{L^2(U, \CC; \mu)}^2 \\
        &\le\sum_{j\in\mathbb N}\|A^{\CC}\|_{H_{\CC,\bsigma}\to L^2(U,\CC;\mu)}^2 \|\tilde v_j\|_{H_{\CC,\bsigma}}^2 
        = \|A^{\CC}\|_{H_{\CC,\bsigma}\to L^2(U, \CC; \mu)}^2\|v\|_{H_{X,\bsigma}}^2,
    \end{align*}
    proving that $\|A^X\|_{H_X\to L^2(X)} \le \|A^{\CC}\|_{H_{\CC,\bsigma}\to L^2(U,\CC;\mu)}$.
    Taking $v=f \otimes \psi_1$ and supremizing over $f\in H_{\CC,\bsigma}$ yields the equality.
    \hfill
\end{proof}

The following theorem plays a key role in transferring current results on linear sampling recovery of functions in RKHSs \cite{BSU23, DKU2023,KUV21,KU21a} to semi-discrete non-intrusive approximation for parametric PDEs with random inputs, based on a finite number of particular solvers.

{Recall that $U_0 \subset U$ is a fixed set of full measure such  that if $v \in H_{X,\bsigma}$, the pointwise evaluations  $v(\by)$ are  well-defined for every $\by \in U_0$.}

\begin{theorem}\label{thm:sameleastsquares}
	Given arbitrary sample points $\by_1,\ldots,\by_n \in U_0$ and functions $h_1,\ldots,h_n \in L_2(U,\CC;\mu)$,	for the sampling algorithm $S_n^X$ in $L_2(U,X;\mu)$ defined by \eqref{S_k}, we have
	\begin{equation} \label{sampling-equality}
		\sup_{v \in B_{X,\bsigma}} \norm{v- S_n^X v}{L_2(U,X;\mu)}
		= \sup_{f \in B_{\CC,\bsigma}} \norm{f - S_n^\CC f}{L_2(U,\CC;\mu)},
	\end{equation}
    and, moreover,
		\begin{equation} \label{rho_n=}
		\varrho_n(B_{X,\bsigma}, L_2(U,X;\mu)) \ = \ \varrho_n(B_{\CC,\bsigma}, L_2(U,\CC;\mu)).
	\end{equation}
\end{theorem}

\begin{proof}
	Denote by $I^X$ the identity operator in $L_2(U,X;\mu)$. Let $S_n^X$ be an arbitrary sampling operator in $L_2(U,X;\mu)$ given for $v \in L_2(U,X;\mu)$ by
	\begin{equation*}
		S_n^X v(\by): = \sum_{i=1}^n v(\by_i) h_i(\by).
	\end{equation*}
    Applying Lemma~\ref{lemma:norms-linear-operators} with $A^X:= I^X - S_n^X$, we get
	\begin{equation*}
		\norm{I^X - S_n^X}{{H_{X,\bsigma}}\to L_2(U,X;\mu)}
		=
		\norm{I^\CC - S_n^\CC}{H_{\CC,\bsigma} \to L_2(U,\CC;\mu)}.
	\end{equation*}
	Consequently, we obtain \eqref{sampling-equality}.

	Since the correspondence between $S_n^X$ and $S_n^\CC$ is one-to-one, 
    we use \eqref{sampling-equality} to show that
	\begin{equation*}
		\inf_{S_n^X \in \Ss_n^X} \ \sup_{v \in B_{X,\bsigma}} \norm{v- S_n^X v}{L_2(U,X;\mu)}
		=
		\inf_{S_n^\CC \in \Ss_n^\CC} \ \sup_{f \in B_{\CC,\bsigma}} \norm{f - S_n^\CC f}{L_2(U,\CC;\mu)},
	\end{equation*}
	which proves \eqref{rho_n=}.
	\hfill
\end{proof}

Let us construct an extension of a least squares approximation in the space $L_2(U,\CC;\mu)$ to a space $L_2(U,X;\mu)$.
{For $n,m\in\NN$ with $n\ge m$}, let $\by_1, \dots, \by_{{n}}\in U_0$ be points, $\omega_1, \dots, \omega_{{n}}\ge 0$ be weights, and $V_m = \operatorname{span}\{\varphi_j\}_{j=1}^{m}$ be the subspace spanned by the functions $\varphi_j$, $j=1,...,m$.
The weighted least squares approximation 
$$
S_{{n}}^{\CC} f 
= 
S_{{n}}^{\CC}(\by_1, \dots, \by_{{n}}, \omega_1, \dots, \omega_{{n}}, V_m) f
$$ 
of a function $f\colon U\to\CC$ is given by
\begin{equation} \label{least-squares-sampling1}
	S_{{n}}^{\CC} f
	= \operatorname{arg\,min}_{g\in V_m} \sum_{i=1}^{{n}} \omega_i |f(\by_i) - g(\by_i)|^2 .
\end{equation}
Here, the $\operatorname{arg\,min}$ is taken as the solution via the Moore-Penrose pseudo-inverse, which selects the minimizer of smallest $L_2$-norm, since there are possibly infinitly many minimizers.
Because $n\ge m$, we are dealing with an over-determined system and the least squares approximation takes the form
\begin{equation} \label{least-squares-sampling2}
	S_{{n}}^\CC f
	= \sum_{s=1}^{m} \hat g_{s} \varphi_s
	\quad\text{with}\quad
	(\hat g_1, \dots, \hat g_{m})^\top
    = (\boldsymbol W^{1/2}\boldsymbol L)^{+} \boldsymbol W^{1/2} (f(\by_1), \dots, f(\by_{{n}}))^{\top},
\end{equation} 
where
 $\boldsymbol L = [\varphi_s(\by_i)]_{i=1,\dots,{n}; \, s=1,\dots,m}$, {$\boldsymbol W = \operatorname{diag}(\omega_1, \dots, \omega_n)$}, and $(\boldsymbol W^{1/2}\boldsymbol L)^{+}$ denotes the pseudo-inverse of $\boldsymbol W^{1/2}\boldsymbol L$.
In the presented theorems the setting is such that the matrix $\boldsymbol W^{1/2}\boldsymbol L$ has full rank, implying the uniqueness of the solutions to the corresponding least squares problems and $(\boldsymbol W^{1/2}\boldsymbol L)^{+} = (\boldsymbol L^\ast\boldsymbol W\boldsymbol L)^{-1}\boldsymbol L^\ast\boldsymbol W^{1/2}$.
For every $n \in \NN$, let
\begin{equation} \label{S_nCC}
	S_{{n}}^\CC f : = \sum_{i=1}^{{n}} f(\by_i) {h_i},
\end{equation}	
be the least squares sampling algorithm constructed as in 
\eqref{least-squares-sampling1}--\eqref{least-squares-sampling2} for these sample points and weights, where 
$h_1,...,h_{{n}} \in V_m$.
Hence, we immediately obtain the extension of this least squares algorithm to the Bochner space $L_2(U,X;\mu)$
by replacing $f\in L_2(U,\CC;\mu)$ with $v\in L_2(U,X;\mu)$:
\begin{equation} \label{S_nX}
	S_{{n}}^X v
	:=
	S_{{n}}^X(\by_1, \dots, \by_{{n}}, \omega_1, \dots, \omega_{{n}}, V_m) v
	:= \sum_{i=1}^{{n}} v(\by_i) h_i.
\end{equation}
As the least squares approximation is a linear operator, worst-case error bounds carry over from the usual Lebesgue space $L_2(U,\CC;\mu)$ to the Bochner space $L_2(U,X;\mu)$.

Let $n \in \NN$ and $E$ be a normed space and $F$ a centrally symmetric compact set in $E$.
Then the Kolmogorov $n$-width of $F$ is defined by
\begin{equation*}
	d_n(F,E):= \ \inf_{L_{n}}\sup_{f\in F}\inf_{g\in L_n}\|f-g\|_E,
\end{equation*}
where the left-most infimum is taken over all subspaces $L_{n}$ of dimension at most $n$ in $E$.
We make use of the abbreviation $d_n := d_n(B_{\CC,\bsigma},L_2(U,\CC;\mu))$.
In our setting, we know that $d_n = \sigma_{n+1}^{-1}$.

From Theorem \ref{thm:sameleastsquares} we can derive the following lemma, which extends to Bochner spaces the recent important result of \cite[Theorem~1]{DKU2023} {on an inequality} between sampling widths and Kolmogorov widths in RKHSs.

\begin{lemma}\label{lemma:leastsquaresbounds(iii)}	
	For any $n\in \NN$, there exist points 
	$\by_1, \dots, \by_n\in U_0$ and weights  $\omega_1, \dots, \omega_n$ such that   
	\begin{equation}\nonumber
		\sup_{v\in B_{X,\bsigma}} \norm{v - \tilde{S}_n^X v}{L_2(U,X;\mu)}^2
        \ \le 433 \max\brab{d_{\lfloor n/43200\rfloor}^2, \, \frac{43200}{n} \sum_{s \ge \lfloor n/43200\rfloor } d_s^2},
	\end{equation}
	where 
    \begin{equation}\label{tilde{S}_n^X:=}
        \tilde{S}_n^X
        := 	
        S_n^X(\by_1, \dots, \by_n, \omega_1, \dots, \omega_n, V_m)
    \end{equation}
    for some $m\in\NN$ satisfying $\frac{n}{43200} \le m \le n$.
\end{lemma}

\begin{proof}
For the particular case when $X = \CC$, this theorem
is implied immediately from \cite[Theorem~23]{DKU2023}. Hence, by 
using Theorem \ref{thm:sameleastsquares} we prove the lemma.
	\hfill
\end{proof}

\subsection{Convergence rates}

{
\begin{lemma} \label{lemma:d_n><}
Let $\|\bsigma^{-1}\|_{\ell_q(\NN)} \le 1$ for some $0 < q \le 2$.	Then we have that
	\begin{equation} \label{d_n><}
		d_n(B_{\CC,\bsigma},L_2(U,\CC;\mu))
        \ \le (n+1)^{-1/q}, \ \ \forall n \in \NN.
	\end{equation}
\end{lemma}	

\begin{proof}	
	For $\xi>0$, we introduce the set 	
	\begin{equation*} \label{Lambda_sigma(xi)}
\Lambda(\xi)
	:= \ 
	\ \big\{s \in \NN: \, \sigma_s^{q} \le \xi \big\}.
\end{equation*}
 For a function $f \in B_{\CC,\bsigma}$
	represented by the series \eqref{series}, we define the truncation 
	\begin{equation} \label{S_{Lambda(xi)}v}
		S_{\Lambda(\xi)} f 
		:= \ 
		\sum_{s\in {\Lambda(\xi)}} f_s \varphi_s.
	\end{equation}	
	Applying Parseval's identity, 
	noting~\eqref{S_{Lambda(xi)}v}, we obtain 
		\begin{equation} \nonumber
		\begin{split}
			\|f- S_{\Lambda(\xi)}f\|_{L_2(U,\CC;\mu)}^2 
			\ & {=} \
            \sum_{s: \ \sigma_{s}> \xi^{1/q} } |f_s|^2 
			\ = \
            \sum_{s: \ \sigma_{s}> \xi^{1/q} } (\sigma_{s}|f_s|)^2 \sigma_{s}^{-2}
			\\[1.5ex]
			\ &\le \
			\xi^{-2/q}
			\sum_{s \in \NN} (\sigma_{s}|f_s|)^2 
			\ \le \ \xi^{-2/q}.
		\end{split}
	\end{equation} 
The function $S_{\Lambda(\xi)}f$ belongs to the linear subspace 
$L(\xi):= \operatorname{span}\{\varphi_s: s \in \Lambda(\xi)\}$ in $L_2(U,\CC;\mu)$ of dimension $|\Lambda(\xi)|$. We have
\begin{equation} \nonumber
	|\Lambda(\xi)|
	\ \le \ 
    \sum_{s: \ \sigma_{s}\le \xi^{1/q} } 1
	\ \le \ 
	\xi \sum_{s \in \NN} \sigma_s^{-q} 
	\ \le \ \xi
\end{equation}
as $\|\bsigma^{-1}\|_{\ell_q(\NN)} \le 1$. 
For a given $n \in \NN$, 
	\begin{equation} \nonumber
			d_n(B_{\CC,\bsigma},L_2(U,\CC;\mu))
			\ \le \
        \sup_{f\in B_{\CC,\bsigma}}\|f- S_{\Lambda(\xi)}f\|_{L_2(U,\CC;\mu)} 
		\ \le \
		\xi^{-1/q}
\end{equation} 
for arbitrary  $\xi < n+1$ satisfying the inequality 
$|\Lambda(\xi)|\le n$.  Hence, by taking the supremum over all such $\xi$ we obtain \eqref{d_n><}.
	\hfill
\end{proof}
\begin{corollary} \label{corollary:varrho_n(iii)}
Let $0 < q < 2$ and $\|\bsigma^{-1}\|_{\ell_q(\NN)} \le 1$.	
    Then there exists a constant $C$, depending on $q$ only, such that, for any $n \in \NN$, there exist points 
		$\by_1, \dots, \by_n \in U_0$ and weights  $\omega_1, \dots, \omega_n$ such that 
	\begin{equation}\nonumber
		\varrho_n(B_{X,\bsigma}, L_2(U,X;\mu)) 
		\ \le \ 
	\sup_{v\in B_{X,\bsigma}} \norm{v - \tilde{S}_{{n}}^X v}{L_2(U,X;\mu)}
		\ \le \ C n^{-1/q},
	\end{equation}
		where  
	$\tilde{S}_n^X$ is defined as in \eqref{tilde{S}_n^X:=}.
	
	Moreover,
		\begin{equation}\nonumber
			(n+1)^{-1/q} 
			\ \le \
			\sup_{\bsigma: \ \|\bsigma^{-1}\|_{\ell_q(\NN) \le 1}}	\varrho_n(B_{X,\bsigma}, L_2(U,X;\mu)) 
            \ \lesssim\ n^{-1/q}.
		\end{equation}
\end{corollary}

 \begin{proof}	
     The upper bound in this theorem is derived from {Lemmata~\ref{lemma:leastsquaresbounds(iii)} and \ref{lemma:d_n><}}, the relation
 	\begin{equation*} \label{rho_n<}
 		\sqrt{\frac{1}{m} \sum_{k \ge m} k^{-2/q}}
 	\ \asymp 	m^{-1/q}, \ \ m \in \NN,
 \end{equation*}	
 and from the fact that this bound is independent of the sequence $\bsigma$.
 To prove the lower bound, one can take $\bsigma=(\sigma_s)_{s \in \NN}$ with 
 $\sigma_s =  (n+1)^{1/q}$ for $s \le n+1$, and 
 $\sigma_s = 2^{Ks/q} $ for $s > n+1$ with arbitrary $K \in \NN$. It is easy to check that  
 $\norm{\tilde{\bsigma}^{-1}}{\ell_q(\NN)} \le 1$, where 
 $\tilde{\bsigma}:= (1+2^{-Kn})^{-1/q}\bsigma$.
 We then have by Theorem \ref{thm:sameleastsquares} that
 \begin{align*}
 	\varrho_n(B_{X,\tilde{\bsigma}}, L_2(U,X;\mu))
 	\ &= \ 
 	\varrho_n(B_{\CC,\tilde{\bsigma}}, L_2(U,\CC;\mu)) 
    \ \ge \
 	d_n(B_{\CC,\tilde{\bsigma}}, L_2(U,\CC;\mu)) 
 \\
     \ &= \
 	\tilde{\sigma}_{n+1}^{-1}
 	= 
 (n+1)^{-1/q}(1+2^{-Kn})^{1/q}.
 \end{align*}
 Since $K$ is arbitrary, we get
 	\begin{equation}\nonumber
 	\sup_{\bsigma: \ \|\bsigma^{-1}\|_{\ell_q(\NN) \le 1}}	\varrho_n(B_{X,\bsigma}, L_2(U,X;\mu)) 
        \ \ge  (n+1)^{-1/q} \,.
 \end{equation}
 	\hfill
 \end{proof}
}

Next, we apply Corollary \ref{corollary:varrho_n(iii)} to Bochner spaces with infinite tensor-product probability measure, which {appear in approximation  of solutions to} parametric PDEs with random inputs and of holomorphic functions in Sections 
 \ref{Applications to parametric PDEs with random inputs} and 
 \ref{Applications to holomorphic functions}, respectively.

 For given $a,b > -1$, let $\nu_{a,b}$ be the Jacobi probability measure on $\II:= [-1,1]$ 
 with the density 
 \begin{equation*} \label{delta_ab}
 	\delta_{a,b}(y):=c_{a,b}(1-y)^a(1+y)^b, \quad
 	c_{a,b}:=\frac{\Gamma(a+b+2)}{2^{a+b+1}\Gamma(a+1)\Gamma(b+1)}.
 \end{equation*}
 Let $(J_k)_{k\in \NN_0}$ be the sequence of Jacobi polynomials on $\II$ 
 normalized with respect to the Jacobi probability measure $\nu_{a,b}$, i.e., 
 $$
 \int_{\II} |J_k(y)|^2 \rd \nu_{a,b}(y) =\int_{\II} |J_k(y)|^2 \delta_{a,b}(y) \rd y =1, \ \
 k\in \NN_0.
 $$
 Let $\gamma$ be the standard Gaussian probability measure on $\RR$ 
 with the density 
 \begin{equation*} \label{g}
 	g(y):=\frac 1 {\sqrt{2\pi}} \mathrm e^{-y^2/2} .
 \end{equation*}
 Let $(H_k)_{k\in \NN_0}$ be the sequence of Hermite polynomials on $\RRR$ 
 normalized with respect to the measure $\gamma$, i.e., 
 $$
 \int_{\RR} |H_k(y)|^2 \rd \gamma(y) =\int_{{\RR}} |H_k(y)|^2 g(y) \rd y =1, \ \
 k\in \NN_0.
 $$
 
 Throughout this section, we use the joint notation: $\UU$ denotes either $\II$ or $\RR$; $\UUi$ either $\IIi$ or $\RRRi$;
 \begin{equation} \nonumber
 	\mu
 	:=
 	\begin{cases}
 		\nu_{a,b}\ \ &{\rm if} \ \UU=\II, \\
 		\gamma\ \ &{\rm if} \ \UU=\RR;
 	\end{cases}
 	\quad
 \end{equation}
 \begin{equation} \nonumber
 	\phi_k
 	:=
 	\begin{cases}
 		J_{{k}}\ \ &{\rm if} \ \UU=\II, \\
 		H_{{k}}\ \ &{\rm if} \ \UU=\RR.
 	\end{cases}
 \end{equation}
 We next recall {the} concept of
 a probability measure $\bmu$ on $\UUi$ as 
 the infinite tensor product of the 
 measures $\mu$:
 \begin{equation} \nonumber
    \mathrm d\bmu(\by) 
 	:= \ 
    \bigotimes_{j \in \NN} \mathrm d \mu(y_j) , \quad \by = (y_j)_{j \in \NN} \in \UUi.
 \end{equation}
 (The sigma algebra for $\bmu$ is generated by the set of cylinders $A:= \prod_{j \in \NN} A_j$, where $A_j \subset \UU$ are univariate Lebesgue measurable sets and only a finite number of $A_i$ are different from $\UU$. For such a set $A$, we have $\bmu(A) = \prod_{j \in \NN} \mu(A_j)$).
 
 Let $X$ be a separable Hilbert space. Then a function $v \in L_2(\UUi,X;\bmu)$ can be represented by the GPC expansion
 \begin{equation} \label{GPCexpansion}
 	v=\sum_{\bs\in\FF} v_\bs \,\phi_\bs, \quad v_\bs \in X,
 \end{equation}
 with
 \begin{equation*}
 	\phi_\bs(\by)=\bigotimes_{j \in \NN}\phi_{s_j}(y_j),\quad 
 	v_\bs:=\int_{\UUi} v(\by)\,\phi_\bs(\by)\, \rd\bmu (\by), \quad 
 	\bs \in \FF.
 \end{equation*}
 
    For a family $\bsigma=(\sigma_\bs)_{\bs \in \FF}$ of positive numbers, denote by $B_{X,\bsigma}(\UUi)$ the set of all functions $v \in L_2(\UUi,X;\bmu)$ represented by 
 	the series \eqref{GPCexpansion} such that
 	\begin{equation*} \label{sigma-summability}
        \left(\sum_{\bs\in\FF} (\sigma_\bs \|v_\bs\|_{X})^2\right)^{1/2} \ \le 1.
 	\end{equation*}
 	
{ Notice that if $\|\bsigma^{-1}\|_{\ell_q(\FF)} < \infty$ for some $0 <q \le 2$, then for every $v \in B_{X,\bsigma}(\UUi)$, the series \eqref{GPCexpansion} converges absolutely and unconditionally in 
 $L_2(\UUi,X;\bmu)$ to $v$ (see \cite[Lemma 3.1]{Dung22} for the case $\UUi = \RRRi$, the case $\UUi = \IIi$ can be proven by the same arguments).
We can reorder the countable set $\FF$ as $\FF = (\bs_j)_{j \in \NN}$ so that the sequence 
 $\bsigma= (\sigma_{\bs_j})_{j \in \NN}$ is non-decreasing. Put 
 $U:=\UUi$, \ $\sigma_j:=\sigma_{\bs_j}$, \ $\varphi_j:=\phi_{\bs_j}$ and $v_j:=v_{\bs_j}$.  Then $B_{X,\bsigma}(\UUi)$ can be seen as the set $B_{X,\bsigma}$, defined as in Section~\ref{sec:intro}, of all functions $v \in L_2(U,X;\bmu)$ represented by 
 the series 
 \begin{equation*} 
 	v= \sum_{j \in \NN}  v_j\,\varphi_j, \quad v_j \in X,
 \end{equation*}
  such that
 \begin{equation*} 
 	\left(\sum_{j \in \NN} (\sigma_j \|v_j\|_{X})^2\right)^{1/2} \ \le 1.
 \end{equation*}
 In the next sections, due to  this representation of $B_{X,\bsigma}(\UUi)$, we 
 are able to employ  Corollary \ref{corollary:varrho_n(iii)} in various applications.}

 	\section{Applications to parametric elliptic PDEs} 
 	\label{Applications to parametric PDEs with random inputs} 
 	
 	\subsection{Introductory remarks} 
 	\label{Introducing remarks}
    Parametric PDEs have a very large or even infinite number of parametric variables. Therefore, such problems are naturally treated as high-dimensional or infinite-dimensional approximation problems.
    As a model problem we consider parametric divergence-form elliptic PDEs with random inputs.
 	
 	Let $D \subset \RRR^d$ be a bounded Lipschitz domain. Consider the diffusion elliptic equation 
 	\begin{equation} \label{ellip}
 		- \dv (a(\bx)\nabla u(\bx))
 		\ = \
 		f (\bx),\ \ \bx \in D,
 		\quad u|_{\partial D} \ = \ 0, 
 	\end{equation}
 	for a {given fixed right-hand side $f \in H^{-1}(D)$ and a 
 	spatially variable scalar diffusion coefficient $a$.
 	Denote by $V:= H^1_0(D)$ the energy space.} If $a \in L_\infty(D)$ satisfies the ellipticity assumption
 	\begin{equation} \nonumber
 		0<a_{\min} \leq a \leq a_{\max}<\infty,
 	\end{equation}
    by the well-known Lax-Milgram lemma, there exists a unique weak
 	solution $u \in V$ to the equation~\eqref{ellip}, satisfying  
 	\begin{equation} \nonumber
 		\int_{D} a(\bx)\nabla u(\bx) \cdot \nabla v(\bx) \, \rd \bx
 		\ = \
 		\langle f , v \rangle, \quad \forall v \in V,
 	\end{equation}
    with $\langle\cdot,\cdot\rangle$ being the duality pairing between $V'=H^{-1}(D)$ and $V=H_0^1(D)$.
 	
 	PDEs with parametric and stochastic inputs are a common model used in science
 	and engineering. Depending on the 
 	nature of the modeled object, the parameters involved in them may be 
    either deterministic or random. The random nature reflects the uncertainty in various parameters,
    which are present in the physical phenomenon modeled by the equation.
 	For equation~\eqref{ellip}, 
 	we consider the diffusion coefficients having a parametric form $a=a(\by)$, where $\by=(y_j)_{j \in \NN}$
    is a sequence of real-valued parameters ranging over the set 
 	$\UUi$ which is either $\RRRi$ or $\IIi$.
 	Denote by $u(\by)$ the solution to the 
 	parametric  elliptic diffusion equation
 	\begin{equation} \label{parametricPDE}
 		- {\rm div} (a(\by)(\bx)\nabla u(\by)(\bx))
 		\ = \
        f {(\bx)}, \ \ \bx \in D, \ \by \in \UUi,
 		\quad u(\by)|_{\partial D} \ = \ 0, \ \by \in \UUi. 
 	\end{equation}	
 	The resulting solution operator maps
 	$\by\in \UUi$ to $u(\by)\in V$. The objective is to 
 	achieve a numerical 
 	approximation of this complex map by a small number of parameters with a
 	guaranteed error in a given norm. 
 	
    We commence with the log-normal case when $\UUi= \RRRi$ and the diffusion coefficient $a$ is of the form
 	\begin{equation} \label{lognormal}
 		a(\by)=\exp(b(\by)), \quad {\text{with }}\ b(\by)=\sum_{j = 1}^\infty y_j\psi_j,
 	\end{equation}
 	and $y_j$ are i.i.d.\ standard Gaussian random 
 	variables, and the affine case when $\UUi= \IIi$ and the diffusion coefficient $a$ is of the form
 	\begin{equation} \label{affine}
 		a(\by)= \bar a + \sum_{j = 1}^\infty y_j\psi_j,
 	\end{equation}
 	and $y_j$ are i.i.d.\ standard Jacobi random 
 	variables.
 	Here $\bar a\in L_\infty(D)$ and $\psi_j \in L_\infty(D)$ {for both  cases.}
 	
 	An approach to studying summability that takes into account the 
 	support properties of the component functions $\psi_j$, has been recently proposed in \cite{BCM17} for the affine parametric case, 
 	in \cite{BCDM17} for {the log-normal parametric case}, 
 	and in \cite{BCDC17} for {extensions} of both cases to 
 	second-order Sobolev norms of the corresponding GPC expansion coefficients. 
 	This approach leads to significant improvements
 	on the results on $\ell_p$-summability and weighted $\ell_2$-summability of GPC expansion coefficients, and therefore, 
 	on best $n$-term semi-discrete and fully discrete approximations 
 	when the component functions $\psi_j$ have limited overlap, 
 	such as splines, finite elements or compactly supported wavelet bases. 
 	In this section, we will employ the results of the previous section {to obtain convergence rates} of sampling recovery of solutions to parametric elliptic PDEs with random inputs, which are derived from results on weighted $\ell_2$-summability in \cite{BCDC17,BCM17}.

 	\subsection{Convergence rates}
 	\label{Convergence rates}
 	{We first present}  some known weighted $\ell_2$-summability 
 	results for solutions $u$ of parametric elliptic PDEs with random inputs.
 	
 	For the log-normal case, we have the following result on weighted $\ell_2$-summability.
 	
 	\begin{lemma}\label{lemma:weighted summablity,lognormal} 
 		Let $0 < q <\infty$, $\eta \in \NN$ with $\eta > 2/q$, and	
 		$\brho:=(\rho_j) _{j \in \NN}$ be a sequence of positive 
 		numbers such that $\brho^{-1} \in \ell_q(\NN)$ and
 		\begin{equation} \label{assumption1}
 			\left\| \sum _{j \in \NN} \rho_j |\psi_j| \right\|_{L_\infty(D)} 
 			<\infty\;.
 		\end{equation} 
 		{Then for the weak solution $u$ to the parametric elliptic PDE \eqref{parametricPDE} with the log-normal  diffusion coefficient as in \eqref{lognormal},} there exist positive constants $M,N$ such that 
 		\begin{equation} \label{w-sum1}
 			\left(\sum_{\bs\in\FF} (\sigma_\bs \| u_\bs\|_{V})^2\right)^{1/2} \le M < \infty\ \ \text{with} \ \
 			\norm{\bsigma^{-1}}{{\ell_q(\FF)}} \le N < \infty,
 		\end{equation}
 		where
 		we define $\bsigma:=\bsigma(\eta,\brho)=(\sigma_\bs(\eta,\brho))_{\bs \in \FF}$ as
 		\begin{equation} \label{sigma_s}
            \sigma_{\bs}^2:= \sigma_{\bs}(\eta,\brho)^2:=\sum_{\bs': \ \sup_{j\in\NN} s_j' \leq \eta}{\bs\choose \bs'} \prod_{j \in \NN}\rho_j^{2s_j'}, \ \ \bs \in \FF.
 		\end{equation}	 		
 	\end{lemma} 
{
 	\begin{proof}
 		By \cite[Theorems 3.3 and 4.2]{BCDM17}, there exists a constant $M$   such that 
 		\begin{equation} 
 			\sum_{\bs\in\FF} (\sigma_\bs \| u_\bs\|_{V})^2
 			=
            \sum_{\bs': \ \sup_{j\in\NN} s_j'\leq \eta} \frac{\brho^{2\bs'}}{\bs'!} 
 				\int_{\RRi}\norm{\partial^{\bs'} u(\by)}{V}^2 \rd \bgamma(\by)
 				\le M.
 		\end{equation}
 This proves the first inequality in \eqref{w-sum1}. Since $\brho^{-1}\in \ell_q(\NN)$.
        The second inequality in \eqref{w-sum1} follows from \cite[Lemma 5.1]{BCDM17}.
 		\hfill
 	\end{proof}
 }

 	\medskip
 	For the affine case, we have the following result on weighted $\ell_2$-summability.
 	\begin{lemma} \label{lemma:weighted summablity,affine} 
 		Let ${\rm ess} \inf \bar a>0$. 	Let $0 < q <\infty$ and	
 		$(\rho_j) _{j \in \NN}$ be a sequence of positive 
 		numbers such that $(\rho_j^{-1}) _{j \in 
 			\NN}$ belongs to $\ell_q(\NN)$ and
 		\begin{equation} \label{assumption2}
 			\left \| \frac{\sum _{j \in \NN} \rho_j|\psi_j|}{\bar a} \right \|_{L_\infty(D)} 
 			< 1.
 		\end{equation} 
 		{Then for the weak solution $u$ to the parametric elliptic PDE \eqref{parametricPDE}  with the affine  diffusion coefficient as in \eqref{affine},} there exist positive constants $M,N$ such that we have that
 		\begin{equation} \label{WS-affine}
 			\left(\sum_{\bs\in\FF} (\sigma_\bs \| u_\bs\|_{V})^2\right)^{1/2} \le M <\infty \ \ \text{with} \ \
 			\norm{\bsigma^{-1}}{{\ell_q(\FF)}} \le N < \infty,
 		\end{equation}
 		where {$\bsigma:=\bsigma(\brho)=(\sigma_\bs(\brho))_{\bs \in \FF}$} is defined by
 		$$
 		\sigma_\bs:=\sigma_\bs (\brho):= \prod _{j \in \NN} c_{s_j}^{a,b}\rho_j^{s_j},
 		$$
        with $c_0^{a,b}:= 1$, and
 		\begin{equation*}
 			c_k^{a,b}
 			:= \
 			\sqrt{\frac{(2k+a+b+1)k! \Gamma(k+a+b+1) \Gamma(a+1) \Gamma(b+1)}
 				{\Gamma(k+a+1)\Gamma(k+b+1)\Gamma(a+b+2)}}, \ k \in \NN.
 		\end{equation*}
 	\end{lemma}
 	
 	{
 		\begin{proof}
The first inequality  in \eqref{WS-affine} follows from \cite[Remark 5.3]{BCM17}, the second one from 	\cite[(63)]{BCDC17}.
 		\hfill
 	\end{proof}
 }
 	
 	\medskip
 	{Notice that assumptions} \eqref{assumption1} and \eqref{assumption2} are different from the assumption 
        {$\brac{\|\psi_j\|_{L_\infty(D)}}_{j \in \NN} \in \ell_p(\NN)$} considered in \cite{CCDS13,CDS10,CDS11}, or 
        {$\brac{j\|\psi_j\|_{L_\infty(D)}}_{j \in \NN} \in \ell_p(\NN)$} considered in \cite{HoSc14} for some $0<p<1$.
        The latter {do not take} into account the support properties of the component functions $\psi_j$, and hence, lead to worse results when the overlaps of the supports of $\psi_j$ are finite. 
 		For a more detailed discussion on the advantages of assumptions \eqref{assumption1} and \eqref{assumption2} over these ones, we refer the reader to \cite{BCM17, BCDM17, Dung21, DNSZ2023}.

 		We are now in a position to formulate the most significant result of our study.
 		Let 
 		$$
 		V_m := \operatorname{span}\{\phi_{\bs_j}\}_{j=1}^{m}
 		$$
 		 be the subspace spanned by the (Hermite or Jacobi) polynomials $\phi_{\bs_j}$, $j=1,...,m$.
 		 {Let $\UUi_0 \subset \UUi$ be a fixed set of full measure such  that the pointwise evaluations  $u(\by)$ are  well-defined for every $\by \in \UUi_0$. (Such a set exists due to Lemma~\ref{lemma:weighted summablity,lognormal} or Lemma~\ref{lemma:weighted summablity,affine}.)}
 	By applying Corollary~\ref{corollary:varrho_n(iii)}, from Lemmata \ref{lemma:weighted summablity,lognormal} and
 	\ref{lemma:weighted summablity,affine}, and utilizing the homogeneous argument we obtain
 	
 	\begin{theorem}\label{theorem:sampling-lognormal-affine(iii)}
 		Let the assumptions and notations of Lemma~\ref{lemma:weighted summablity,lognormal} or 
 		of Lemma~\ref{lemma:weighted summablity,affine} with $0 < q < 2$ hold for the log-normal case \eqref{lognormal} $(\UUi = \RRi)$ or for the affine case \eqref{affine} $(\UUi = \IIi)$, respectively. 
 {Let  $u$ be the weak solution to the parametric elliptic PDE \eqref{parametricPDE}  with the log-normal diffusion coefficient as in \eqref{lognormal} or the affine diffusion coefficient as in 
 	\eqref{affine}, respectively.}  Then for any $n \in \NN$, there exist points 
$\by_1, \dots, \by_n \in \UUi_0$ and weights  $\omega_1, \dots, \omega_n$ such that 
	\begin{equation*}
	\norm{u - \tilde{S}_n^V u}{L_2(\UUi,V;\bmu)}
	\ \le C MN n^{-1/q}
\end{equation*}
    with a constant $C$ independent of $n,M,N$, and $u$,
	where  
$\tilde{S}_n^V$ is defined as in \eqref{tilde{S}_n^X:=} for $X=V$.
 	\end{theorem}		
 	
 	\section{Applications to holomorphic functions} 
\label{Applications to holomorphic functions}

{The sparsity analysis for parametric elliptic PDEs with log-normal diffusion coefficients, as in \cite{BCDC17, BCDM17}, hinges on real-variable bootstrapping to establish sparsity. This approach encounters technical obstacles when extending to higher spatial regularity or more general parametric PDEs. By contrast, complex-variable methods proposed in \cite{DNSZ2023}, using holomorphic extension of the solution offer an alternative pathway to sparsity and regularity, often simplifying the treatment of smoothness and broadening applicability.
One advantage of
establishing sparsity of Hermite GPC expansion coefficients via
holomorphy rather than by successive differentiation is that it allows
to derive, in a unified way, weighted $\ell_2$-summability bounds for the 
coefficients of a Hermite GPC expansion whose magnitudes are measured in terms
Sobolev scales on the domain $D$.

Formally, in the log-normal case \eqref{lognormal} of the parametric equation \eqref{parametricPDE},
replacing $\by=(y_j)_{j \in \NN}\in \RRi$ in the coefficient $a(\by)$ in \eqref{lognormal} by
$\bz=(z_j)_{j \in \NN}=(y_j+i \xi_j)_{j \in \NN}\in \CC^\infty$, 
the real part of $a(\bz)$ is
\begin{equation*} \label{Re(a)} \mathfrak{R}[a(\bz)] =
	\exp\Bigg({\sum_{j \in \NN} y_j\psi_j}\Bigg) \cos\Bigg(\sum_{j
		\in \NN} \xi_j\psi_j\Bigg)\,.
\end{equation*}
We find that $\mathfrak{R}[a(\bz)]>0$ if
$$
\bigg\|\sum_{j \in \NN} \xi_j\psi_j \bigg\|_{L_\infty(D)} <
\frac{\pi}{2}.
$$
 This motivates the study
of the analytic continuation of the solution map $\by \mapsto u(\by)$
to $\bz \mapsto u(\bz)$ for complex parameters
$\bz = (z_j)_{j \in \NN}$ where each
$z_j$ lies in the strip
\begin{equation} \label{eq:DefSjrho} 
	\mathcal{S}_j (\brho):= \{ z_j\in
	\CC\,: |\mathfrak{Im}z_j| < \rho_j\}
\end{equation}
and where $\rho_j>0$ and
$\brho=(\rho_j)_{j \in \NN}$ is any sequence of
positive numbers such that
\begin{equation*} \label{k-011} \Bigg\|\sum_{j\in \NN} \rho_j
	|\psi_j|\Bigg\|_{L_\infty(D)} < \frac{\pi}{2}\,.
\end{equation*}

Let $\brho=(\rho_j)_{j\in \NN}$ be a sequence of
non-negative numbers and assume that $J \subseteq \supp(\brho)$ is
finite. Define for $\by \in \RRi$,
{
\begin{equation}\label{eq:Snubrho}
	\mathcal{S}_J (\by,\brho) := \big\{(z_j)_{j\in \NN} \in \CCi: z_j \in
	\mathcal{S}_j(\brho)\ \text{if}\ j\in J, \ \text{and}\ z_j=y_j \
	\text{if}\ j\not \in J \big\}.
\end{equation}
}
For the definition of Sobolev spaces $H^r:= H^r(D)$ and $W^r_\infty:= W^r_\infty(D)$ {as well as that of a} $C^m$-domain see, e.g., \cite{Adams2003}.
The following result on holomorphy of the parametric solution has been proven in \cite[Proposition 3.21]{DNSZ2023}.
\begin{lemma}\label{lemma:holoh1}
		{Let $r \in \NN$ and $D$ be a bounded domain with either $C^\infty$-boundary or convex $C^{r-1}$-boundary.	}
	Let the sequence $\brho=(\rho_j)_{j\in \NN}\in [0,\infty)^\infty$
	satisfy
\begin{equation}\label{eq:leqkappa}
		\Bigg\|\sum_{j \in \NN} \rho_j |\psi_j | \Bigg\|_{L_\infty(D)} 
		\leq 
		\kappa < \frac{\pi}{2}\,.
	\end{equation}
	Let $\by_0=(y_{0,1},y_{0,2},\ldots) \in {\RRi}$ be such that $b(\by_0)$
	belongs to $W^{r-1}_\infty$, and let $J\subseteq \supp(\brho)$ be a
	finite set.
	Then the weak parametric solution $u$ of the variational form of \eqref{parametricPDE} with log-normal random inputs \eqref{lognormal} is
	holomorphic on $ \mathcal{S}_J (\by_0,\brho) $ as a function of the
	parameters
	$\bz_J=(z_j)_{j \in \NN} \in \mathcal{S}_J (\by_0,\brho)$
	taking values in $H^r(D)$ with $z_j = y_{0,j}$ for $j\not \in J$ held
	fixed.
\end{lemma}

Based on the holomorphy of the parametric solution as in Lemma \ref{lemma:holoh1}, a weighted $\ell_2$-summability of the Sobolev $H^r$-norm of the {Hermite GPC expansion} coefficients of the parametric solution $u$ has been established in 
\cite[Theorem 3.25]{DNSZ2023} as follows.

\begin{lemma} \label{lemma:H^r} 
Let $r \in \NN$, $D$ be a bounded domain with either $C^\infty$-boundary or convex $C^{r-1}$-boundary, {and $f \in H^{r-1}(D)$.}	Assume that for every $j\in \NN$, $\psi_j\in W^{r-1}_\infty$, and there
	exists a positive sequence $(\lambda_j)_{j\in \NN}$ such that
	$(\exp(-\lambda_j^2))_{j\in \NN}\in \ell_1(\NN)$ and the
	series $\sum_{j\in \NN}\lambda_j|D^{\balpha}\psi_j|$ converges in
	$L_\infty(D)$ for all $\balpha \in \NN_0^d$ with ${|\balpha| \le r -1}$.
	Let 
	$\bvarrho=(\varrho_j)_{j\in \NN}$
	be a sequence of positive numbers satisfying
	$(\varrho_j^{-1})_{j \in \NN}\in \ell_q(\NN)$ for some
	$0 < q < \infty$. Assume that, for each
	$\bs\in \FF$, there exists a sequence 
	$\brho_\bs= (\rho_{\bs,j})_{j \in \NN}$ of non-negative numbers such that
	$\supp(\bs)\subseteq \supp(\brho_\bs)$,
	\begin{equation} \label{assumption: theorem 3.1}
		\sup_{\bs\in \FF} \sum_{|\balpha| \le s -1}\Bigg\| \sum_{j\in
			\NN}\rho_{\bs,j}|D^{\balpha}\psi_j|\Bigg\|_{L_\infty(D)}\leq
		\kappa <\frac{\pi}{2}, \qquad \text{and} \qquad
		\sum_{\bs: \ |\bs|_\infty\leq \eta}
		\frac{\bs!\bvarrho^{2\bs}}{\brho_\bs^{2\bs}}
		<\infty
	\end{equation} 
	with $\eta \in \NN$, $\eta > 2/q$. {Let  $u$ be the weak solution to the parametric elliptic PDE \eqref{parametricPDE}  with the log-normal diffusion coefficient as in \eqref{lognormal}.}
    Then there exist positive constants $M,N$ such that
	\begin{equation} \label{eq:beta-u} 
		\sum_{\bs\in\FF} \brac{\sigma_\bs(\eta,\bvarrho)\|u_\bs\|_{H^r}}^2 \le M <\infty \ \ \text{with} \ \
		\norm{\bsigma(\eta,\bvarrho)^{-1}}{\ell_q(\FF)} \le N < \infty,		
	\end{equation}
    where
    $\bsigma(\eta,\bvarrho)=(\sigma_\bs(\eta,\bvarrho))_{\bs \in \FF}$ is given by \eqref{sigma_s}.
\end{lemma}

By applying Corollary \ref{corollary:varrho_n(iii)}, from Lemma \ref{lemma:H^r}, under the assumptions and notation of Lemma~\ref{lemma:H^r}, and utilizing the homogeneous argument we again obtain 

\begin{corollary}\label{corollary:sampling-lognormal}
	Let the assumptions and notation of Lemma \ref{lemma:H^r} hold  for some
	$0 < q < 2$.  {Let  $u$ be the weak solution to the parametric elliptic PDE \eqref{parametricPDE}  with the log-normal diffusion coefficient as in \eqref{lognormal}.}
		Then for any $n \in \NN$, there exist points 
		$\by_1, \dots, \by_n \in \UUi_0$ and weights  $\omega_1, \dots, \omega_n$ such that 
	\begin{equation} \label{convergence-rate-H^r}
		\norm{u - \tilde{S}_n^{H^r} u}{L_2(\RRi,H^r;\bgamma)}
		\ \le C MN n^{-1/q}
	\end{equation}
	with a constant $C$ independent {of} $n,M,N$ and $u$,
	 where 
	$\tilde{S}_n^{H^r}$ is defined as in \eqref{tilde{S}_n^X:=} for $X= H^r$.
	\end{corollary}
    Note that the convergence rate in \eqref{convergence-rate-H^r} of Corollary \ref{corollary:sampling-lognormal} holds for every $r\in \NN$, and, moreover, in the case $r=1$ significantly improves the previously known rate from \cite[Corollary~5.9]{Dung21} by a factor of $n^{-1/2}$.
	
The results of Lemmata \ref{lemma:holoh1} and \ref{lemma:H^r} encourage us to investigate the holomorphy and weighted $\ell_2$-summability as a sequence for a wider class of functions on $\RRi$ and application to approximation for parametric PDEs with log-normal random inputs.
We recall the concept of ``$(\bb,\xi,\varepsilon,X)$-holomorphic
functions'' on $\RRi$ which has been introduced {in \cite[Definition 4.1]{DNSZ2023}} for general parametric PDEs with random input data. 
For $m\in\NN$ and a positive sequence $\bvarrho=(\varrho_j)_{j=1}^m$, we put
\begin{equation*}
	\label{eq:Sjrho}
	\Ss(\bvarrho) := \set{\bz\in \CC^m}{|\mathfrak{Im}z_j| < \varrho_j~\forall j}\qquad\text{and}\qquad
	\Bb(\bvarrho) := \set{\bz\in\CC^m}{|z_j|<\varrho_j~\forall j}.
\end{equation*}

Let $X$ be a complex separable Hilbert space,
$\bb=(b_j)_{j\in\NN}$ a positive sequence, and $\xi>0$, $\varepsilon>0$.
For $m\in\NN$ we say that a positive sequence $\bvarrho=(\varrho_j)_{j=1}^m$ is
\emph{$(\bb,\xi)$-admissible} if
\begin{equation*}\label{eq:adm}
	\sum_{j=1}^m b_j\varrho_j\leq \xi\,.
\end{equation*}
A function $v\in L_2(\RRi,X;\gamma)$ is called
$(\bb,\xi,\varepsilon,X)$-holomorphic if
\begin{enumerate}
	\item[{\rm (i)}]\label{item:hol} for every $m\in\NN$ there exists
	$v_m:\RR^m\to X$, which, for every $(\bb,\xi)$-admissible
	$\bvarrho$, admits a holomorphic extension
	(denoted again by $v_m$) from $\Ss(\bvarrho)\to X$; furthermore,
	for all $m<m'$
	\begin{equation*}\label{eq:un=um}
		v_m(y_1,\dots,y_m)=v_{m'}(y_1,\dots,y_m,0,\dots,0)\qquad\forall (y_j)_{j=1}^m\in\RR^m,
	\end{equation*}
	
	\item[{\rm (ii)}]\label{item:varphi} for every $m\in\NN$ there exists
	$\varphi_m:\RR^m\to\RR_+$ such that
	$\norm{\varphi_m}{L_2(\RR^m;\bgamma)}\le\varepsilon$ and
\begin{equation*} \label{ineq[phi]}
		\sup_{\text{{$\bvarrho$} is $(\bb,\xi)$-adm.}}~\sup_{\bz\in
			\Bb(\bvarrho)}\norm{v_m(\by+\bz)}{X}\le
		\varphi_m(\by)\qquad\forall\by\in\RR^m,
	\end{equation*}
	\item[{\rm (iii)}]\label{item:vN} with $\tilde v_m:\RRRi\to X$ defined by
	$\tilde v_m(\by) :=v_m(y_1,\dots,y_m)$ for $\by\in \RRRi$ it holds
	\begin{equation*}
		\lim_{m\to\infty}\norm{v-\tilde v_m}{ L_2(\RRi, X;\bgamma)}=0.
	\end{equation*}
\end{enumerate}

We mention some important examples of $(\bb,\xi,\varepsilon,X)$-holomorphic functions on $\RRi$ which are solutions to parametric PDEs with log-normal random inputs and which were studied in \cite{DNSZ2023}. 
Let $b(\by)$ be defined as in \eqref{lognormal} and $\Vv$ a holomorphic map from an open set in $L_\infty(D)$ to $X$. Then function compositions of the type 
$$
v(\by)= \Vv(\exp(b(\by)))
$$
are $(\bb,\xi,\varepsilon,X)$-holomorphic under certain conditions \cite[Proposition 4.11]{DNSZ2023}.
The map $\mathcal V$ could be the solution operator, which maps the diffusion coefficient $a$ to the solution $\mathcal V(a)$.
In particular in the example above, $\mathcal V(\exp(b(\by)))$ is exactly the parametric solution to the parametric diffusion equation with log-normal input discussed in Section \ref{Applications to parametric PDEs with random inputs}.
The abstraction allows to consider structurally similar PDEs with log-normal random inputs as well.
This allows us to apply weighted $\ell_2$-summability {for  approximation} of solutions $v(\by)= \Vv(\exp(b(\by)))$ as $(\bb,\xi,\varepsilon,X)$-holomorphic functions on various function spaces $X$, to a wide range of parametric and stochastic PDEs with log-normal inputs.
Such function spaces $X$ include, for example, high-order regularity spaces $H^s(D)$ \cite[Section 4.3.1]{DNSZ2023} and corner-weighted Sobolev (Kondrat'ev) spaces $K^s_\varkappa(D)$ ($s \ge 1$) for the parametric elliptic PDEs \eqref{ellip} with log-normal inputs \eqref{lognormal} \cite[Section 7.6.1]{DNSZ2023}, spaces of solutions to linear parabolic PDEs with log-normal inputs \eqref{lognormal} \cite[Section 4.3.2]{DNSZ2023}, spaces of solutions to linear elastics equations with log-normal modulus of elasticity 
\cite[Section 4.3.3]{DNSZ2023}, spaces of solutions to Maxwell equations with log-normal permittivity \cite[Section 4.3.4]{DNSZ2023}; spaces of posterior densities and of their linear functionals in Bayesian inverse problems \cite[Section 5]{DNSZ2023}.

The following key result on weighted $\ell_2$-summability of $(\bb,\xi,\varepsilon,X)$-holomorphic functions has been proven in \cite[ Theorem 4.9]{DNSZ2023}.
\begin{lemma} \label{lemma:weighted summability, holomorphic} Let $v$ be
	$(\bb,\xi,\varepsilon,X)$-holomorphic for some $\bb\in \ell_p(\NN)$ with $0< p <1$. Let $\eta\in\NN$ and let the sequence $\brho=(\rho_j)_{j \in \NN}$ be defined by
	$$
{\rho_j:=b_j^{p-1}\frac{\xi}{4\sqrt{\eta!}\norm{\bb}{\ell_p(\NN)}}.}
	$$
Then we have 
	\begin{equation*} \label{ell_2-summability}
		\left(\sum_{\bs\in\FF} (\sigma_\bs \|v_\bs\|_{X})^2\right)^{1/2} \ \le M \ <\infty, \ \ \text{with} \ \
		\norm{\bsigma^{-1}}{\ell_q(\FF)} \le N < \infty,
	\end{equation*}	
	where $q := p/(1-p)$,
{$\bsigma:=\bsigma(\eta,\brho)=(\sigma_\bs(\eta,\brho))_{\bs \in \FF}$ is given by \eqref{sigma_s}, 
	$M= \varepsilon C_{\bb}$ and
	$N= C_{\bb,\xi}$  with some positive constants $C_{\bb}$ and $C_{\bb,\xi}$.}
\end{lemma}

{Let $\RRi_0 \subset \RRi$ be a fixed set of full Gaussian measure such  that the pointwise evaluations  $u(\by)$ are  well-defined for every $\by \in \RRi_0$. (Such a set exists due to 
	Lemma~\ref{lemma:weighted summability, holomorphic}.)} By applying Corollary \ref{corollary:varrho_n(iii)}, from Lemma \ref{lemma:weighted summability, holomorphic}, and utilizing the homogeneous argument we obtain
 
 \begin{theorem}\label{theorem:sampling-lognormal-holomorphic} 
 	Let $v$ be
 	$(\bb,\xi,\varepsilon,X)$-holomorphic for some $\bb\in \ell_p(\NN)$ with 
 	$0< p <1$. 
     Then for any $n \in\NN$, there exist points 
$\by_1, \dots, \by_n \in \RRi_0$ and weights  $\omega_1, \dots, \omega_n$ such that 
 	\begin{equation} \label{bound1}
 		{\norm{v - \tilde{S}_n^X v}{L_2(\RRi,X; \bgamma)}}
 		\ \le C MN n^{-(1/p - 1)}
 	\end{equation}
 	with a constant $C$ independent of $n$ and {$v$},  	
 	where 
 	 $M,N$ are as in Lemma \ref{lemma:weighted summability, holomorphic} and
 	 $\tilde{S}_n^X$ is defined as in \eqref{tilde{S}_n^X:=}. 
 	\end{theorem}

    The convergence rate in Theorem \ref{theorem:sampling-lognormal-holomorphic} notably improves the result \cite[Theorem 6.13]{DNSZ2023} by a factor of $n^{-1/2}$.
 		
 	We present two examples of application of Theorem~\ref{theorem:sampling-lognormal-holomorphic} to parametric PDEs with random inputs. 
 First, let us revisit the parametric equation \eqref{parametricPDE} with 	log-normal random inputs \eqref{lognormal}. Hence, if
 $(\psi_j)_{j\in\NN}\subset L_\infty(D)$ such that $\bb\in\ell_1(\NN)$ with
 $b_j:=\norm{\psi_j}{L_\infty(D)}$, 
 then the parametric weak solution $u(\by)$ uniquely exists and
is {$(\bb,\xi,\varepsilon,V)$}-holomorphic by \cite[Theorem 4.11]{DNSZ2023}. For details, see \cite[Section 4.3.1]{DNSZ2023}. 
 	Assume in addition, that $\bb$ is a  sequence such that $\bb\in\ell_p(\NN)$
 for some $0 < p <2/3$. Then the parametric solution $u(\by)$ can be approximated by the extended least squares sampling algorithm $\tilde{S}_n^V$ defined as in Theorem \ref{theorem:sampling-lognormal-holomorphic} for $X=V$, for which we have the convergence rate \eqref{bound1}.
 	
Second, we apply Theorem \ref{theorem:sampling-lognormal-holomorphic} to solutions to parametric linear parabolic PDEs with log-normal random inputs. {To  establish this, let us  present the holomorphic properties of solutions of these equations.  For details, see \cite[Section 4.3.2]{DNSZ2023}.} 
 Let $0<T<\infty$ denote a finite time-horizon and let $D$ be a
 bounded domain with Lipschitz boundary $\partial D$ in $\RR^d$. We
 define $I:=(0,T)$ and consider the initial boundary value problem
 for the linear parabolic PDE 
\begin{equation}\label{eq:parabolic}
 	\begin{cases}
 		\frac{\partial u(t,\bx)}{\partial t} 
 		- {\rm div}\big(a(\bx)\nabla u(t,\bx)\big)=f(t,\bx), \qquad (t,\bx)\in I\times D,
 		\\
 		u|_{\partial D\times I}=0,
 		\\
 		u|_{t=0}=u_0(\bx).
 	\end{cases}
 \end{equation}
We write $V:=H_0^1(D;\CC)$ and $V':=H^{-1}(D;\CC)$. Let
\begin{equation} \label{X:=}
 	X := L_2(I,V)\cap H^1(I,V') = \big(L_2(I)\otimes V \big)\cap \big(H^1(I)\otimes V'\big)
 \end{equation}
 equipped with the sum norm
 \begin{equation}\label{normX:=}
 	\|u\|_{X} := \Big(\|u\|_{L_2(I,V)}^2 +\|u\|_{H^1(I,V')}^2 \Big)^{1/2},\quad u\in X,
 \end{equation}
 where
 $$
 \|u\|_{L_2(I,V)}^2 = \int_I \|u(t, \cdot) \|_V^2\, \rd t\, ,
 $$
 and
 $$
 \|u\|_{H^1(I,V')}^2 = \int_I \| \partial_t u(t, \cdot) \|_{V'}^2\, \rd
 t\,.
 $$
 To state a space-time variational formulation and to specify the data
 space for \eqref{eq:parabolic}, we introduce the test-function space
\begin{equation*}\label{eq:ParIBVPY}
 	Y 
 	= L_2(I,V)\times L_2(D) 
 	= \big(L_2(I)\otimes V \big) \times L_2(D)
 \end{equation*}
which we endow with the norm
\begin{equation*}
 	\|v\|_{Y}
 	=
 	\Big(\|v_1\|_{L_2(I,V)}^2 + \|v_2\|_{L_2(D)}^2\Big)^{1/2},\quad v=(v_1,v_2)\in Y\,.
 \end{equation*}
 Given a time-independent diffusion coefficient $a\in L_\infty(D;\CC)$
 and $(f,u_0)\in Y'$, the continuous sesqui-linear and anti-linear forms
 corresponding to the parabolic problem \eqref{eq:parabolic} read for
 $u\in X$ and $v=(v_1,v_2)\in Y$ as
 \begin{equation*}
 	\begin{split}
 		B(u,v)& := \int_I \int_D \partial_t u\,\overline{v_1}\rd \bx
 		\rd t + \int_I\int_D a\nabla u \cdot \overline{\nabla v_1} \rd
 		\bx \rd t + \int_D u_0\,\overline{v_2 } \rd \bx
 	\end{split}
 \end{equation*}
 and
 \begin{equation*}
 	\begin{split}
 		L(v) := \ \int_I \big\langle f(t,\cdot),v_1(t,\cdot) \big\rangle
 		\rd t + \int_D u_0 \,\overline{v_2 } \rd \bx,
 	\end{split}
 \end{equation*}
where $\langle \cdot,\cdot \rangle$ is the anti-duality pairing
 between $V'$ and $V$. Then the space-time variational (weak) formulation of
 equation \eqref{eq:parabolic} is: Find a weak solution $u \in X$ such that
\begin{equation}\label{para-weak}
 	B(u,v)=L(v),\quad \forall v\in Y\,.
 \end{equation}
 	Assume that $(f,u_0)\in Y'$ and that
 \begin{equation}\label{eq:para-uni}
 	0 < \rho(a) := \underset{\bx\in D}{\operatorname{ess \, inf}}\,\Re(a(\bx))
 	\leq |a(\bx)|
 	\leq \|a\|_{L_\infty(D)}
     <\infty,\qquad \bx\in D {.}
 \end{equation}
 Then there exists a unique solution $u$ to the equation
 \eqref{para-weak}.
 
 Consider the initial boundary value problem for the parametric linear parabolic PDE 
\begin{equation}\label{eq:p-parabolic}
	\begin{cases}
		\frac{\partial u(\by)(t,\bx)}{\partial t} 
		- {\rm div}\big(a(\by)(\bx)\nabla u(\by)(t,\bx)\big)=f(t,\bx), \quad (t,\bx)\in I\times D, \ \ \by \in \RRi,
		\\[1ex]
		u(\by)(\bx)|_{\partial D\times I}=0, \quad \by \in \RRi,
		\\[1ex]
		u(\by)(\bx)|_{t=0}=u_0(\bx), \quad \by \in \RRi,
	\end{cases}
\end{equation}
with log-normal inputs \eqref{lognormal}. 
If
$(\psi_j)_{j\in\NN}\subset D$ such that $\bb\in\ell_1(\NN)$ with
$b_j:=\norm{\psi_j}{L_\infty(D)}$, 
then the parametric weak solution $u(\by)$ uniquely exists and is {$(\bb,\xi,\varepsilon,X)$}-holomorphic by \cite[Theorem 4.11]{DNSZ2023}. 	
Assume in addition, that $\bb$ is a  sequence such that $\bb\in\ell_p(\NN)$
for some $0 < p <2/3$. Then the parametric weak solution $u(\by)$ to the equation~\eqref{eq:p-parabolic} can be approximated by the extended least squares sampling algorithm $\tilde{S}_{n}^X$ defined in Theorem~\ref{theorem:sampling-lognormal-holomorphic} for  the space $X$ as in \eqref{X:=} and \eqref{normX:=}, for which we have the convergence rate \eqref{bound1}.

\section{Constructiveness and alternative least squares methods}
\label{Constructiveness and alternative least squares methods}

In this section, we present several different sampling schemes to use with the least squares algorithms for functions in the RKHS $H_{\CC,\bsigma}$, and inequalities between sampling $n$-widths and Kolmogorov $n$-widths of the unit ball $B_{\CC,\bsigma}$. We explain then how to apply these inequalities to obtain corresponding convergence rates of linear sampling recovery in abstract Bochner spaces and {of  approximation} of the parametric solution $u(\by)$ to parametric elliptic or parabolic PDEs with log-normal or affine inputs, as well as of infinite-dimensional holomorphic functions.

The choice of points $\by_1, \dots, \by_n$, weights $\omega_1, \dots, \omega_n$, and approximation space $V_m$ is crucial for the error of the least squares approximation.
A lot of work has been done in the usual Lebesgue space $L_2(U,\CC;\mu)$ of which we present two more choices with a trade-off between constructiveness and tightness of the bound and transfer them to the Bochner space $L_2(U,X;\mu)$.

For $m\in\NN$, let the probability measure $\nu = \nu(m)$, introduced in \cite{KU21a}, be defined by
\begin{equation}\label{nu}
	\mathrm d\nu(\by)
	:= \varrho(\by)\mathrm d\mu(\by)
	:= \frac{1}{2}\left( \frac{1}{m}\sum_{s=1}^{m}|\varphi_s(\by)|^2
	+ \frac{\sum_{s=m+1}^{\infty}|\sigma_{s}^{-1}\varphi_s(\by)|^2}
	{\sum_{s=m+1}^{\infty}\sigma_{s}^{-2}} \right)\mathrm d\mu(\by) .
\end{equation}

{
\begin{assumption}\label{points}
    Let $m\ge 2$.
	\begin{itemize}
		\item[\rm (i)]
        Let $n = \lceil 40 m\log m\rceil$.
            Let further $\by_1  ,\dots, \by_{n}\in U_0$ be points drawn i.i.d.\ with respect to $\nu$ and set $\omega_i := (\varrho(\by_i))^{-1}$.
		\item[\rm (ii)]
		Let $\lceil 40 m\log m\rceil$ points be drawn i.i.d.\ with respect to $\nu$ and subsampled using \cite[Algorithm~3]{BSU23} to $n = \lceil bm\rceil$ points for some $b>1+1/m$.
        Denote the resulting points by $\by_1, \dots, \by_{n} \in U_0$ and $\omega_i = (\varrho(\by_i))^{-1}$.
	\end{itemize}
\end{assumption}

Note that (i) and (ii) in Assumption \ref{points} have been considered in \cite{KU21a} and 
\cite{BSU23}.
Recall that we use the abbreviation $d_n := d_n(B_{\CC,\bsigma},L_2(U,\CC;\mu))$.

\begin{lemma}\label{lemma:leastsquaresbounds}
	Let 
	$S_{{n}}^X 
	:=
	S_{{n}}^X(\by_1, \dots, \by_{{n}}, \omega_1, \dots, \omega_{{n}}, V_m) $ 
	 be the extended least squares algorithm defined as in \eqref{S_nX}.
    Then we have the following.
	\begin{enumerate}
		\item[\rm (i)]
        The points and weights from Assumption~\ref{points}{\rm (i)} fulfill with probability exceeding $1-(160\log n)/n$
        \begin{equation} \nonumber
            \sup_{v\in B_{X,\bsigma}} \norm{v - S_{n}^X v}{L_2(U,X;\mu)}
            \le 19 \max \Big\{ d_{\lfloor n/(40\log n)\rfloor}, 
            \sqrt{\frac{\log n}{n} \sum_{s \ge \lfloor n/(40\log n)\rfloor} d_s^2 } \Big\} \,.
        \end{equation}
		\item[\rm (ii)]
        The points and weights from Assumption~\ref{points}{\rm (ii)} fulfill with probability exceeding $1-8b/n$
		\begin{equation}\nonumber
			\sup_{v\in B_{X,\bsigma}} \norm{v - S_{n}^X v}{L_2(U,X;\mu)}
            \le 600  \Big(\frac{b+1}{b-1}\Big)^{3/2} 
            \max \Big\{ \sqrt{\log n} \ d_{\lfloor n/(2b)\rfloor}, \sqrt{ \frac{\log n}{n} 
            	\sum_{s\ge \lfloor n/(2b)\rfloor} d_s^2 } \Big\} \,.
		\end{equation}
		\item[\rm (iii)]
        There exist points $\by_1, \dots, \by_n\in U_0$ and weights $\omega_i\ge 0$ such that
		\begin{equation}\nonumber
			\sup_{v\in B_{X,\bsigma}} \norm{v - S_{n}^X v}{L_2(U,X;\mu)}
            \ \le  4325 \max\Big\{ d_{\lfloor n/43200\rfloor}, \sqrt{\frac{1}{n} 
            	\sum_{s \ge \lfloor n/43200\rfloor} d_s^2 } \Big\}.
		\end{equation}
	\end{enumerate}
\end{lemma}

\begin{proof}
    By Theorem \ref{thm:sameleastsquares}, it suffices to show the inequalities for $X = \CC$.
    To prove (i) we use \cite[Theorem~7.7]{barteldiss} with $t=\log m$.
    An intermediate step in the proof states with probability exceeding $1-4\exp(-t)$
    \begin{equation} \nonumber
        \sup_{v\in B_{X,\bsigma}} \norm{v - S_{n}^X v}{L_2(U,X;\mu)}
        \le \sqrt{ 3 d_{m}^2 + \frac{168 \log n}{n} \sum_{s\ge m+1} d_s^2 }
        \le 19 \max\Big\{ d_{m}, \sqrt{ \frac{\log n}{n} \sum_{s\ge m+1} d_s^2 } \Big\} \,.
    \end{equation}
    From the assumption on $n$ we have $ n/(40\log n) \le m$, which yields the first assertion and the desired probability $1-4\exp(-t) = 1-4/m \ge 1-(160\log n)/n$.

    To prove (ii) we use \cite[Theorem~7.8]{barteldiss} with $t=\log m$.
    An intermediate step in the proof states with probability exceeding $1-4\exp(-t) = 1-4/m \ge 1-8b/n$
    \begin{align*}
        \sup_{v\in B_{X,\bsigma}} \norm{v - S_{n}^X v}{L_2(U,X;\mu)}
        &\le \sqrt{ 89 \frac{(b+1)^2}{(b-1)^3} \frac{\lceil 40 m \log m\rceil}{m}
        \Big( 3 d_{m}^2 + \frac{168 \log \lceil 40 m\log m\rceil }{\lceil 40 m\log m\rceil } \sum_{s\ge m+1} d_s^2 \Big) } \\
        &\le \sqrt{ 3596 \frac{(b+1)^2}{(b-1)^3} \log m
        \Big( 3 d_{m}^2 + \frac{168 \log \lceil 40 m\log m\rceil }{\lceil 40 m\log m\rceil } \sum_{s\ge m+1} d_s^2 \Big) } \,.
    \end{align*}
    With $m\ge 2$ we have
    \begin{equation*}
        \frac{168 \log \lceil 40 m\log m\rceil }{\lceil 40 m\log m\rceil }
        \le \frac{168 \log ( 41 m\log m) }{40 m\log m}
        \le \frac{25}{m} \,.
    \end{equation*}
    Thus
    \begin{align*}
        \sup_{v\in B_{X,\bsigma}} \norm{v - S_{n}^X v}{L_2(U,X;\mu)}
        &\le \sqrt{ 3596 \frac{(b+1)^2}{(b-1)^3} \log m
        \Big( 3 d_{m}^2 + \frac{25}{m} \sum_{s\ge m+1} d_s^2 \Big) } \\
        &\le \sqrt{ 359600 \Big(\frac{b+1}{b-1}\Big)^3 \log n
        \max\Big\{ d_{\lfloor n/(2b)\rfloor}^2, \frac{1}{n} \sum_{s\ge \lfloor  n/(2b)\rfloor} d_s^2 \Big) \Big\} } \,.
    \end{align*}
    
    Part (iii) of the assertion is given in \cite[Theorem~23]{DKU2023} 
    (see Lemma \ref{lemma:leastsquaresbounds(iii)}).
	\hfill
\end{proof}
}

Regarding the constructiveness of the linear sampling algorithms in Lemma~\ref{lemma:leastsquaresbounds}, the bound in Lemma~\ref{lemma:leastsquaresbounds}(i) is the coarsest bound, but the points construction requires only a random draw, which is computationally inexpensive.
The sharper bound in Lemma~\ref{lemma:leastsquaresbounds}(ii) uses an additional constructive subsampling step.
This was implemented and numerically tested in \cite{BSU23} for up to 1000 basis functions.
For larger problem sizes the current algorithm is too slow as its runtime is cubic in the number of basis functions.
The sharpest bound in Lemma~\ref{lemma:leastsquaresbounds}(iii) is a pure existence result.
So, the only way to obtain this point set is to brute-force every combination, which is computational infeasible.

Similarly to the proof of Corollary \ref{corollary:varrho_n(iii)}, one can prove the following results on convergence rates of the extended least squares sampling algorithms described in Lemma \ref{lemma:leastsquaresbounds}.

\begin{corollary} \label{corollary:varrho_n}
	Let $0< q < 2$ and {$\|\bsigma^{-1}\|_{\ell_q(\NN)} \le 1$}. 
		Let 
	$S_{{n}}^X 
	:=
	S_{{n}}^X(\by_1, \dots, \by_{{n}}, \omega_1, \dots, \omega_{{n}}, V_m) $ 
	be the extended least squares algorithm defined as in \eqref{S_nX}.
	There are universal constants $c_1, c_2, c_3 \in \NN$ such that for all $n \ge 2$ we have the following.
	\begin{enumerate}
		\item[\rm (i)]
        The points and weights from Assumption~\ref{points}{\rm (i)} fulfill with probability at least $1-(160\log n)/n$
		\begin{equation}\nonumber
			\sup_{v\in B_{X,\bsigma}} \norm{v - S_{c_1n}^X v}{L_2(U,X;\mu)}
            \ \lesssim n^{-1/q} (\log n)^{1/q}.
		\end{equation}
		\item[\rm (ii)]
            The points and weights from Assumption~\ref{points}{\rm (ii)} fulfill with probability at least $1-(160\log n)/n$
		\begin{equation}\nonumber
			\sup_{v\in B_{X,\bsigma}} \norm{v - S_{c_2n}^X v}{L_2(U,X;\mu)}
			\ \lesssim n^{-1/q}  (\log n)^{1/2}.
	\end{equation}
	\item[\rm (iii)]
    There exist points $\by_1, \dots, \by_n\in U_0$ and weights $\omega_i\ge 0$ such that
	\begin{equation}\nonumber
		\sup_{v\in B_{X,\bsigma}} \norm{v - S_{c_3n}^X v}{L_2(U,X;\mu)}
		\ \lesssim n^{-1/q}.
	\end{equation}
	\end{enumerate}	
\end{corollary}	

One can apply the last corollary to parametric PDEs with random inputs and infinite-dimensional holomorphic functions to receive counterparts of all the results in Sections \ref{Applications to parametric PDEs with random inputs} and \ref{Applications to holomorphic functions} on extended least squares sampling recovery, based on the sample points from 
Assumption~\ref{points}{\rm (i)}--{\rm (iii)}.
In particular, by applying Corollary \ref{corollary:varrho_n}, from Lemmata \ref{lemma:weighted summablity,lognormal} and
\ref{lemma:weighted summablity,affine} we obtain:

\begin{corollary}\label{corollary:sampling-lognormal-affine}
	Let 
	$S_{{n}}^V 
	:=
	S_{{n}}^V(\by_1, \dots, \by_{{n}}, \omega_1, \dots, \omega_{{n}}, V_m) $ 
	be the extended least squares algorithm defined as in \eqref{S_nX} for $X=V$.
	Under the assumption  of Theorem \ref{theorem:sampling-lognormal-affine(iii)},
	there are universal constants $c_1, c_2,c_3\in \NN$ such that for all $n \ge 2$ we have the following.
	\begin{enumerate}
		\item[\rm (i)]
        The points and weights from Assumption~\ref{points}{\rm (i)} fulfill with probability at least $1-(160\log n)/n$
		\begin{equation*}
			\norm{u - S_{c_1n}^V u}{L_2(\UUi,V;\bmu)}
			\ \le C MN n^{-1/q} (\log n)^{1/q}.
		\end{equation*}
		\item[\rm (ii)]
        The points and weights from Assumption~\ref{points}{\rm (ii)} fulfill with probability at least $1-(160\log n)/n$
		\begin{equation*}
			\norm{u - S_{c_2n}^V u}{L_2(\UUi,V;\bmu)}
			\ \le C MN n^{- 1/q} (\log n)^{1/2}.
		\end{equation*}	
			\item[\rm (iii)]
        There exist points $\by_1, \dots, \by_n\in U_0$ and weights $\omega_i\ge 0$ such that
		\begin{equation*}
			\norm{u - S_{c_3n}^V u}{L_2(\UUi,V;\bmu)}
			\ \le C MN n^{- 1/q}.
		\end{equation*}	
	\end{enumerate}	 	
	The constants $C$ in the above inequalities are independent of $M,N,n$ and $u$.
\end{corollary}

We compare the results in Theorem \ref{theorem:sampling-lognormal-affine(iii)}  Corollary~\ref{corollary:sampling-lognormal-affine} with the known results in prior works.
In the affine case \eqref{affine}, the convergence rate $(n/\log n)^{-1/q}$ (with $1/q:= 1/p - 1/2$) in terms of the number $n$ of sampling points has been received in \cite{CCMNT2015} (see also \cite{CM2018}) for an  \emph{adaptive} least squares approximation ``lifted"  to Hilbert-valued functions, based on an $\ell_q$-summability of the Legendre GPC expansion coefficients of the parametric solution, and on an adaptive choice of a sequence of finite dimensional approximation spaces, which is different from the linear extended least squares approximation in Corollary~\ref{corollary:sampling-lognormal-affine}(i). Notice also that the result in Corollary~\ref{corollary:sampling-lognormal-affine}(i) for the affine case could be also proven by a linear modification of the technique used in \cite{CCMNT2015}, based on the weighted $\ell_2$-summability \eqref{WS-affine}. 
The convergence rate $(n/\log n)^{-1/q}$ (with $1/q:= 1/p - 1/2$) of sampling recovery of $(\bb,\varepsilon)$-holomorphic functions on $\IIi$ has been proven in \cite{ADM2024} based on the least squares procedure of \cite{KU21} ``lifted" to Hilbert-valued functions, where $\varepsilon >0$, $\bb=\brac{b_j}_{j \in \NN} \in \ell_p(\NN)$ is a positive sequence and $0<p<1$. The results presented in this paper for the affine case \eqref{affine} in Theorem~\ref{theorem:sampling-lognormal-affine(iii)} and Corollary~\ref{corollary:sampling-lognormal-affine} improve upon those in the cited literature by a logarithmic factor of $(\log n)^{-1/q}$.

 \medskip
 \noindent
 {\bf Acknowledgments:} 
The work of Dinh D\~ung is funded by the Vietnam National Foundation for Science and Technology Development (NAFOSTED) in the frame of the NAFOSTED--SNSF Joint Research Project under Grant 
IZVSZ2$_{ - }$229568. 
Felix~Bartel acknowledges the financial support from the Australian Research Council Discovery Project DP240100769.
 A part of this work was done when the authors were working at the Vietnam Institute for Advanced Study in Mathematics (VIASM). They would like to thank the VIASM for providing a fruitful research environment and working condition.
 
\bibliographystyle{abbrv}

\end{document}